\documentclass{article}
\usepackage[utf8]{inputenc}
\usepackage[T1]{fontenc}
\usepackage{graphicx}
\usepackage{epstopdf}
\usepackage{amssymb}
\usepackage{latexsym}
\usepackage{listings}
\usepackage{amsfonts}
\usepackage{amsthm}
\usepackage{bm}
\usepackage{amsmath}
\usepackage{thmtools}
\usepackage[linesnumbered]{algorithm2e}
\usepackage{biblatex}
\graphicspath{ {Figures/} } 
\usepackage{color}
\usepackage{standalone} 
\usepackage[table,x11names]{xcolor}
\usepackage{soul}
\usepackage{multirow}
\usepackage{multicol}			
\usepackage{caption}			
\usepackage{subcaption}		
\usepackage[colorlinks]{hyperref}
\hypersetup{citecolor = {blue},pdfauthor=author}
\usepackage{cleveref}
\usepackage{tikz}
\usepackage{tikz-cd}
\usepackage{pgfplots}
\usepackage{pgfplotstable}
\usepackage{sectsty}			
\usepackage{physics}
\usepackage{array}
\usepackage{float}
\usepackage{forest}
\usepackage{xparse}
\usepackage{lipsum}
\usepackage{multicol}
\usepackage{todonotes}
\usepackage{csvsimple}

\newcommand{\citeabstract}[1]{[\citeauthor{#1}, \citefield{#1}{journaltitle}, \citeyear{#1}]}

\usepackage{acro}

\acsetup{first-style=long-short,
    barriers/use=true,
    barriers/reset=true,
    }

\DeclareAcronym{sbp}        {short=SBP, long=summation by parts}
\DeclareAcronym{sat}        {short=SAT, long=simultaneous approximation term}
\DeclareAcronym{ibp}        {short=IBP, long=integration by parts}
\DeclareAcronym{ibvp}       {short=IBVP,long=initial boundary value problem}
\DeclareAcronym{cg}         {short=CG,  long=conjugate gradient}
\DeclareAcronym{gmres}      {short=GMRES,long=generalised minimium residual}
\DeclareAcronym{bicgstab}   {short=BiCG-Stab,long=biconjugate gradient stabilised}
\DeclareAcronym{imex}       {short=IMEX,long=implicit explicit}
\DeclareAcronym{cfl}        {short=CFL, long=Courant–Friedrichs–Lewy}
\DeclareAcronym{ode}        {short=ODE, long=ordinary differential equation}
\DeclareAcronym{pde}        {short=PDE, long=partial differential equation}
\DeclareAcronym{cn}         {short=CN,  long=Crank-Nicholson}
\DeclareAcronym{fem}        {short=FEM, long=finite element method}
\DeclareAcronym{fd}         {short=FD,  long=finite difference}
\DeclareAcronym{dg}         {short=DG,  long=discontinuous Galerkin}
\DeclareAcronym{spec}       {short=SPEC,long=Stepped Pressure Equilibrium Code}
\DeclareAcronym{ap}         {short=AP,  long=asymptotic preserving}
\DeclareAcronym{kam}        {short=KAM, long=Kolmogorov-Arnol'd-Moser}
\DeclareAcronym{idw}        {short=IDW, long=Inverse Distance Weighting}

\usetikzlibrary{shapes,arrows,arrows.meta,decorations.pathmorphing,decorations.markings,decorations.pathreplacing,patterns,shapes.geometric}
\usepgfplotslibrary{groupplots}
\pgfplotsset{compat=1.13}
\DeclareMathOperator{\arcsinh}{asinh}
\pgfkeys{/pgf/declare function={arcsinh(\x) = ln(\x + sqrt(\x^2+1));}}

\tikzcdset{scale cd/.style={every label/.append style={scale=#1},
    cells={nodes={scale=#1}}}}

\newcommand{\R}{\mathbb{R}}

\newcommand{\half}{\frac{1}{2}}

\newcolumntype{C}{>{$}c<{$}}

\newtheorem{theorem}{Theorem}[section]
\theoremstyle{definition}

\declaretheoremstyle[
  numbered=yes,
  numberlike=theorem,
  spaceabove=1em plus 0.75em minus 0.25em,
  spacebelow=1em plus 0.75em minus 0.25em,
  qed={}
]{exmpstyle}

\newtheorem{lemma}[theorem]{Lemma}
\newtheorem*{remark}{Remark}

\newcommand{\logLogSlopeTriangleText}[6]
{

    \pgfplotsextra
    {
        \pgfkeysgetvalue{/pgfplots/xmin}{\xmin}
        \pgfkeysgetvalue{/pgfplots/xmax}{\xmax}
        \pgfkeysgetvalue{/pgfplots/ymin}{\ymin}
        \pgfkeysgetvalue{/pgfplots/ymax}{\ymax}

        \pgfmathsetmacro{\xArel}{#1}
        \pgfmathsetmacro{\yArel}{#3}
        \pgfmathsetmacro{\xBrel}{#1-#2}
        \pgfmathsetmacro{\yBrel}{\yArel}
        \pgfmathsetmacro{\xCrel}{\xArel}

        \pgfmathsetmacro{\lnxB}{\xmin*(1-(#1-#2))+\xmax*(#1-#2)} 
        \pgfmathsetmacro{\lnxA}{\xmin*(1-#1)+\xmax*#1} 
        \pgfmathsetmacro{\lnyA}{\ymin*(1-#3)+\ymax*#3} 
        \pgfmathsetmacro{\lnyC}{\lnyA+(-#4)*(\lnxA-\lnxB)}
        \pgfmathsetmacro{\yCrel}{\lnyC-\ymin)/(\ymax-\ymin)} 

        \coordinate (A) at (rel axis cs:\xArel,\yArel);
        \coordinate (B) at (rel axis cs:\xBrel,\yBrel);
        \coordinate (C) at (rel axis cs:\xCrel,\yCrel);

        \draw[#5]   (A) node[pos=0.5,anchor=north] {1}
                    (B)-- 
                    (C) node[xshift=0.2,yshift=0.2,pos=0.6,above,sloped,#6] {slope #4};
    }
}

\newcommand{\logLogSlopeTriangle}[5]{\logLogSlopeTriangleText{#1}{#2}{#3}{#4}{#5}{}}

\newcommand\keywords[1]{%
    \begingroup
    \let\and\\
    \par
    \noindent\textit{Keywords:} #1\par
    \endgroup
}

\newcommand{\setfoot}[2]{%
    \footnote{#2}%
    \newcounter{#1}%
    \setcounter{#1}{\value{footnote}}%
}

\usepackage[margin=20mm]{geometry} 

\usepackage{biblatex}
\usepackage{color}
\usepackage{soul}
\usepackage{todonotes}
\pgfplotsset{
        compat=1.13,
        my style/.style={
            width=0.5\textwidth,height=0.5\textwidth,
            xlabel=$n$,
            ylabel=$\log(rel\; err)$,
        },
        my other style/.style={
            width=0.5\textwidth,height=0.3\textwidth,
            xlabel=$n$,
        },
        my legend style/.style={
            legend entries={
                2nd order,
                4th order,
            },
            legend style={
                at={([yshift=4pt]1.0,1.1)},
                anchor=north,
            },
            legend columns=2,
        },
        cycle multiindex* list={
            blue!75!black,
            red!75!black,
            green!75!black,
            cyan!75!black,
            magenta!75!black,
            yellow!75!black,
            black!75!black,
                \nextlist
            mark=*,
            mark=square*,
            mark=x,
                \nextlist
        }}

\newcommand{\equal}{=}

\title{A provably stable numerical method for the anisotropic diffusion equation in confined magnetic fields: Curvilinear coordinates and multi-block domains.}
\author{D. Muir 
    \setfoot{cor}{Corresponding author} 
    \setfoot{nmpp}{Numerical Methods in Plasma Physics, Max Plank Institute for Plasma Physics, Garching.}, 
    K. Duru
    \setfoot{utep}{The University of Texas at El Paso, Texas, United States.},
    S. Hudson
    \setfoot{proxima}{Proxima Fusion GmbH, Floßergasse 2, Munich, 81369, Germany.},
    M. Hole
    \setfoot{anu}{Mathematical Sciences Institute, Australian National University, Australian Capital Territory, \textsc{Australia}.} 
    }
\date{February 2024}

\begin{document}

\maketitle


\begin{abstract}
    We present a robust and accurate numerical method for the anisotropic diffusion equation in curvilinear coordinates. This study extends the recent work \citeabstract{muir_provably_2025} for solving the anisotropic diffusion equation in magnetic fields from Cartesian meshes to to curvilinear coordinates and complex geometries.
    The method uses summation by parts with simultaneous approximation terms for computing the diffusion perpendicular to field lines. The diffusion along field lines is computed using a penalty approach, similar to a simultaneous approximation term, but applied across the volume.
    To extend the method to complex geometry we use a multi-block approach with piecewise smooth structured meshes. That is, the domain is split into sub-grids, with locally adjacent boundaries coupled weakly using penalties.
    We prove the semi-discrete stability for the curvilinear implementation by deriving discrete energy estimates. 
    The approach is verified though a number of numerical tests, which demonstrate the convergence properties of the method in multi-domain approach.
    Finally, we present a qualitative result generated in complex geometry and magnetic field, which is generated by the Stepped Pressure Equilibrium Code.
\end{abstract}

\section{Introduction}

The anisotropic diffusion equation can be used for modelling heat transport in magnetic confinement fusion plasmas. In fusion scenarios, the ratio of diffusion coefficient parallel and perpendicular to the magnetic field can be on the order of $10^{10}$ \cite{fitzpatrick_helical_1995}, and introduces disparate temporal and spatial scales in the solution. The presence of the disparate scales in the solution, spanning many orders of magnitude, creates a significant challenge for numerical methods.
In addition, fusion simulations often require modelling complex geometries, for instance when simulating transport in the edge region of a plasma between regions of the main plasma body and the wall of a fusion device. These regions typically feature magnetic islands, separatrices or field line chaos, which may be used to control the impact of fast particles on the wall of the fusion device.
The complexity of the magnetic geometry adds to the complication of the strongly multi-scale problem, since the structure of the magnetic field may prevent the construction of a grid which can be aligned with it. 
This can result in the slow scale perpendicular diffusion being polluted by numerical errors from the fast scale \cite{hudson_temperature_2008,chacon_asymptotic-preserving_2014,gunter_modelling_2005, umansky_numerical_2005}.

Recent work \cite{muir_provably_2025} introduced a novel penalty method to solving the anisotropic diffusion equation in Cartesian coordinates. Since the interest is in toroidal confinement fusion, the penalty approach is suitable in cases where the magnetic field is periodic in at least one direction. The method uses field line tracing to model the diffusion along field lines, and a non-linear penalty term that weakly implements the parallel diffusion.

In this study we extend the results of that work \cite{muir_provably_2025} from Cartesian meshes to curvilinear geometries.
To facilitate the implementation in geometrically complex domains, the computational grid is split into multiple sub-domains, and the solution is approximated locally. The numerical solution across domain boundaries is coupled weakly using simultaneous approximation terms, in a way equivalent to certain discontinuous Galerkin schemes \cite{gassner_skew-symmetric_2013}. The variable coefficient SBP operators introduced in previous work efficiently compute the curvilinear derivatives with the metric information embedded in the diffusion coefficient matrix.
Further building on that work, the method in curvilinear coordinates is also shown to be unconditionally stable.
We remark that recent work \cite{chacon_asymptotic-preserving_2024} has shown that non-splitting methods can achieve larger time steps in the presence of a separatrix due to the separation in time scales. However, current approach still relies on operator splitting due to the construction of the parallel operator.

The paper is structured as follows; section \ref{sec:Continuous analysis} outlines the field aligned form of the anisotropic diffusion equation.
This is followed in section \ref{sec:Numerical preliminaries} by introducing the \ac{sbp} operators and the parallel penalty, though we refer readers to \cite{muir_provably_2025} for a more elaborate introduction.
In section \ref{sec:semidiscrete analysis} we perform analysis of the semi-discrete problem, including deriving an energy bound for the problem with curvilinear coordinates. 
This is followed by analysis of the fully discrete problem in \ref{sec:Discrete analysis}, though the majority of the details of the time stepping scheme are given in \cite{muir_provably_2025}.
Verification of the numerics is presented in section \ref{sec:Verification}, this includes verification of the perpendicular only operator by manufactured solution, and a self-convergence test with a single island.
We then present a numerical experiment by utilising magnetic fields from the \ac{spec} for MHD equilibria.
Finally, we conclude in section \ref{sec:Conclusion}.

\section{Continuous analysis}\label{sec:Continuous analysis}

Theoretical discussion on the diffusion equation can be found in many texts, see for example \cite{evans_partial_2010}. Further, the field aligned anisotropic diffusion equation has also been discussed by many authors \cite{gunter_modelling_2005,chacon_asymptotic-preserving_2014,degond_asymptotic-preserving_2012,van_es_finite-volume_2016,helander_heat_2022}. Therefore, to ensure a self-contained narrative, we only briefly reintroduce the problem here.

Let $u:\Omega\times[0,T]\to\R$ be the solution to the anisotropic diffusion equation, which for instance may correspond to temperature,
\begin{align}\label{eq:ADE}
    \pdv{u}{t} = \nabla\cdot(\bm{K}\nabla u),
\end{align}
where $\Omega\subset\R^3$ and $T>0$ and where $K\in\R^{3\times3}_{\geq0}$ is the symmetric matrix of diffusion coefficients.
Letting the magnetic field $\bm{B}:\Omega\to\R^3$ be a vector valued function then the diffusion coefficient  matrix can be written,
\begin{align}\label{eq:ADE diffusion coefficient matrix}
    \bm{K} = \kappa_\perp I + (\kappa_\parallel - \kappa_\perp)\frac{\bm{B} \bm{B}^T}{|\bm{B}|^2},
\end{align}
where $|\bm{B}|^2 = \bm{B}_x^2 + \bm{B}_y^2 + \bm{B}_z^2$. Defining $\nabla_\parallel := \bm{B}(\bm{B}^T \nabla)/|\bm{B}|^2$, then one can write the equation in field aligned form as,
\begin{align}\label{eq:ADE cont laplace form}
    \pdv{u}{t} = \nabla\cdot(\kappa_\perp\nabla u + \tilde{\kappa}_\parallel\nabla_\parallel u),
\end{align}
where $\tilde{\kappa}_\parallel=\kappa_\parallel-\kappa_\perp$. Unless otherwise stated, for the rest of the paper we will drop the tilde so that $\tilde{\kappa}_\parallel = \kappa_\parallel$.
The diffusion equation \eqref{eq:ADE} is augmented with Dirichlet boundary conditions,
\begin{align}\label{eq:ADE cont boundary conditions}
    u(x_\Gamma,y_\Gamma,\theta,t) = g(x_\Gamma,y_\Gamma), \qquad x_\Gamma, y_\Gamma\in \partial\Omega
\end{align}
and smooth initial condition,
\begin{align}\label{eq:ADE cont initial condition}
    u(x,y,\theta,0) = u_0(x,y,\theta).
\end{align}
For scalar functions $u$ and $v$ we define the inner product and $L^2$ norm by,
\begin{align}
    (u,v) = \int_\Omega u v \, \dd x, \qquad 
    \|u\|^2 = \int_\Omega u^2 \,\dd x.
\end{align}
The anisotropic diffusion equation can then be shown to have the well known energy bound.

\begin{theorem}\label{thrm:ADE cont stability}
    Consider the \acl{ibvp} given by equation \eqref{eq:ADE cont laplace form} with initial condition \eqref{eq:ADE cont initial condition} and boundary conditions \eqref{eq:ADE cont boundary conditions}. If $g=0$ then
    \begin{align}
        \|u\|^2 \leq \|u_0\|^2.
    \end{align}
\end{theorem}
The proof of which can be found in standard texts on partial differential equations. A stable numerical method should mimic the energy bound in Theorem \ref{thrm:ADE cont stability}.

While it is often possible to run simulations in Cartesian coordinate systems, it may be that the magnetic fields provided by a plasma equilibrium are in more complex coordinate systems, or provided on a grid in real space. Therefore, to facilitate simulations in more general and real geometries, in this study we extend the perpendicular component of the anisotropic diffusion equation \eqref{eq:ADE cont laplace form} to curvilinear coordinates.

\subsection{Curvilinear coordinates}

Now consider the smooth coordinate transformation $(x(q,r),y(q,r)) \iff (q(x,y),r(x,y))$ where $(q,r)\in\Omega_c\subset[0,1]\times[0,1]$. Moreover suppose that the perpendicular diffusion coefficient changes in space so that $u\left(x(q,r),y(q,r)\right)$ and $\kappa_\perp\left(x(q,r),y(q,r)\right)$.

The Jacobian determinant is given by $J = x_q x_r - y_q y_r > 0$ and the metric relations are,
\begin{align}
    J q_x=y_r, \qquad J q_y=-x_r, \qquad J r_x=-y_q, \qquad J r_y=x_q.
\end{align}
where the subscripts on the coordinates denote the metric derivatives.
Then in curvilinear coordinates the perpendicular component of the anisotropic diffusion equation is,
\begin{align}\label{eq:ADE continuous curvilinear}
    J\pdv{u}{t} = 
        \pdv{q} \left( \kappa_q \pdv{u}{q} \right) + 
        \pdv{r} \left( \kappa_r \pdv{u}{r} \right) + 
        \pdv{q} \left( \kappa_{qr} \pdv{u}{r} \right) +
        \pdv{r} \left( \kappa_{rq} \pdv{u}{q} \right),
\end{align}
where,
\begin{align}
\begin{aligned}
    &\kappa_q = J\left( q_x^2 + q_y^2 \right) \kappa_\perp ,\quad
    \kappa_r = J\left( r_x^2 + r_y^2 \right) \kappa_\perp ,\\
    &\kappa_{qr} = \kappa_{rq} = J\left(q_xr_x + q_yr_y \right).
\end{aligned}
\end{align}

Note that with the appropriate boundary conditions equation \eqref{eq:ADE continuous curvilinear} satisfies a similar energy estimate as described in Theorem \eqref{thrm:ADE cont stability}.

\section{Numerical preliminaries}\label{sec:Numerical preliminaries}

We introduce the numerical techniques first by discussing the operators in one spatial dimension. 
Consider the 1D interval $x\in[x_L,x_R]$, $x_L<x_R$ discretised into $n$ points, so that
$$
\bm{x}_i = x_L + (i-1)\Delta x, \qquad \Delta x = (x_L - x_R)/(n-1), \qquad i = 1,\dots,n.
$$
Further, let $\bm{u}(t)=[u(x_1,t),u(x_2,t),\dots,u(x_n,t)]^T\in\R^n$ denote the semi-discrete form of the scalar field on the grid.
Now let $D_x,D_{xx}^{(\kappa)}\in\R^n$ be the discrete approximations to the first and second derivative operators so that
\begin{align}
    \left.\pdv{u}{x}\right|_{x_i}\approx (D_x\bm{u})_i, \qquad 
    \left.\pdv{x}\left(k\pdv{u}{x}\right)\right|_{x_i} \approx (D_{xx}^{(k)}\bm{u})_i,
\end{align}
where $k>0$.

Let $H=H^T>0$ define a inner product, norm and composite quadrature rule with strictly positive weights such that,
\begin{align}\label{eq:H inner product norm}
    \int u v \dd x \approx \bm{u}^T H \bm{v},\qquad 
    \|\bm{u}\|_H = \sqrt{\bm{u}^T H \bm{u}}.
\end{align}
Then the first derivative operator is called a \acl{sbp} operator if it can be written,
\begin{align}\label{eq:1D First derivative SBP operator}
    D_x = H^{-1}Q, \qquad Q + Q^T = E_n - E_1 = \operatorname{diag}([-1,0,\dots,0,1]),
\end{align}
such that $E_1 = \operatorname{diag}(\begin{bmatrix} 1 , 0 , \dots , 0 \end{bmatrix})$ and $E_n = \operatorname{diag}(\begin{bmatrix} 0, \dots, 0, 1 \end{bmatrix})$ and where the entries of $Q$ must be determined given accuracy and order constraints.
Let $E_n - E_1 = B = \operatorname{diag}(-1,0,\dots,0,1)$ be the diagonal matrix which projects onto the domain boundaries such that $B\bm{u}=[-u_1,0,\dots,0,u_N]$. Then $D_{xx}^{(\kappa)}$ is called a fully compatible variable coefficient \ac{sbp} operator \cite{mattsson_summation_2012,duru_stable_2014} if it can be written,
\begin{align}
    D_{xx}^{(k)} = H^{-1}(-M^{(k)} + BKD_x), \qquad M^{(k)} = D_x^THKD_x + R_x,
\end{align}
where $R_x=R_x^T\geq 0$ is the remainder term and $K=\operatorname{diag}([\kappa(x_1),\kappa(x_2),\dots,\kappa(x_n)])$ is a matrix of the diffusion coefficient evaluated at the grid points. Note that $M^{(k)}$ is symmetric and positive definite.

For further discussion we refer readers to previous work \cite{muir_provably_2025} or other works such as \cite{mattsson_summation_2012,del_rey_fernandez_review_2014} for a more in-depth description of the \ac{sbp} operators.

\section{The numerical parallel operator}\label{sec:sub:numerical parallel operator}

To define the two dimensional parallel diffusion operator first consider the curvilinear domain $\Omega_c=[0, 1]\times[0, 1]$ as before. Let the coordinates $q$ and $r$ be discretised into $n_q$ and $n_r$ points respectively, so that
\begin{align*}
\bm{q}_i = \frac{(i-1)}{n_q - 1}, \qquad
\bm{r}_j = \frac{(j-1)}{n_r -1 }, \qquad
i = 1,2,\dots,n_q,\quad j = 1,2,\dots,n_r.
\end{align*}
where $\bm{q}_i$ and $\bm{r}_j$ correspond to $(x_i,y_j)$ physical coordinates by way of coordinate transform.
In each of the coordinates the mass matrix $H$ can be written,
\begin{align}\label{eq:Curvilinear H}
    H_q = \Delta q\operatorname{diag}([h_1,h_2,\dots,h_{n_q}]), \qquad 
    H_r = \Delta r\operatorname{diag}([h_1,h_2,\dots,h_{n_r}]),
\end{align}
such that,
\begin{align}
    \bm{H} = H_q\otimes H_r, \qquad \bm{1}^T\bm{H}J\,\bm{u} \approx \int_\Omega u \,\dd x\,\dd y,
\end{align}
where $\bm{1}$ is a vector of ones. 

From \cite{muir_provably_2025} the numerical parallel operator is given by,
\begin{align}\label{eq:Numerical parallel opeartor}
    \bm{P}_\parallel = -\tau_\parallel \kappa_\parallel \bm{H}^{-1} \left(\bm{I} - \frac{1}{2}[\bm{P}_f + \bm{P}_b]\right).
\end{align}
Here, $\mathbf{P}_f$ and $\mathbf{P}_b$ are stable interpolation operators which approximate the solution at points forward and backwards along the magnetic field by field line tracing. The full operator $\bm{P}_\parallel$ then approximates the parallel diffusion.
For stability the parallel penalty must have $\tau_\parallel>0$.
We can be more specific by considering,
\begin{align}
    \bm{w} = \bm{w}_f + \bm{w}_b = \half\left(\mathbf{P}_f\bm{u} + \mathbf{P}_b\bm{u}\right),
\end{align}
then defining $\tau_\parallel$ as,
\begin{align}
    \tau_\parallel = \alpha\left( \frac{\|\bm{u}-\bm{w}\|}{\|\bm{u}\|} \right)^{\beta},
\end{align}
where $\alpha,\beta\in\R$. Previous emperical studies found that $\alpha=0.1$, $\beta=2$ provided good convergence properties while reducing errors due to interpolation \cite{muir_provably_2025}.

\section{Semi-discrete analysis}\label{sec:semidiscrete analysis}

The semi-discrete form of the curvilinear anisotropic diffusion equation \eqref{eq:ADE continuous curvilinear} can then be written,
\begin{align}
    \dv{\bm{u}}{t} = (\bm{D}_\perp + \bm{P}_\parallel)\bm{u},
\end{align}
where
\begin{align}\label{eq:Discrete Pperp Pparallel}
    \bm{D}_\perp \approx \nabla\cdot(\kappa_\perp \nabla), \quad\text{and}\quad 
    \bm{P}_\parallel \approx \nabla \cdot (\kappa_\parallel \nabla_\parallel).
\end{align}
The parallel operator is the same as in previous work \cite{muir_provably_2025}.
We first turn our attention to the perpendicular isotropic diffusion.

\subsection{The perpendicular operator}

To define $\bm{D}_\perp$ consider the two dimensional first derivative \ac{sbp} and boundary operators which are extended from \eqref{eq:1D First derivative SBP operator} by Kronecker product,
\begin{align}\label{eq:2D SBP Operators}
\begin{aligned}
    &\bm{D}_q = D_q\otimes I_r, \quad 
        \bm{H}_q = H_q\otimes I_r, \quad
        \bm{B}_q = B_q\otimes I_r, \\
    &\bm{D}_r = I_q\otimes D_r, \quad 
        \bm{H}_r = I_q\otimes H_r, \quad
        \bm{B}_r = I_q\otimes B_r.
\end{aligned}
\end{align}
where $I_q\in\R^{n_q\times n_q}$ and $I_r\in\R^{n_r\times n_r}$ are identity matrices.
Using these, the second derivative operators in curvilinear coordinates can be defined as,
\begin{align}
    \bm{D}_{qq}^{(\bm{K}_q)} &= \bm{H}_q^{-1}(-\bm{D}_q^T \bm{H}_q \bm{K}_q \bm{D}_q + \bm{B}_q\bm{K}_q\bm{D}_q + \bm{R}_q) \\
    \bm{D}_{rr}^{(\bm{K}_r)} &= \bm{H}_r^{-1}(-\bm{D}_r^T \bm{H}_r \bm{K}_r \bm{D}_r + \bm{B}_r\bm{K}_r\bm{D}_r + \bm{R}_r)
\end{align}
where the $\bm{R}_q$ and $\bm{R}_r$ matrices are the remainder matrices and the coefficient matrices are,
\begin{align}\label{eq:discrete curvilinear diffusion matrix}
\begin{array}{lll}
        &\bm{K}_q    = J(q_x^2 \bm{K}_x + q_y^2 \bm{K}_y),
            &\bm{K}_r    = J(r_x^2 \bm{K}_x + r_y^2 \bm{K}_y), \\
        &\bm{K}_{qr} = J(q_xr_x \bm{K}_x + q_yr_y \bm{K}_y), 
            &\bm{K}_{qr} = \bm{K}_{qr}^T.
\end{array}
\end{align}
Given the above definitions, the discrete form of the perpendicular diffusion operator \eqref{eq:Discrete Pperp Pparallel} in curvilinear coordinates is defined as,
\begin{align}\label{eq:perpendicular FD operator}
    \nabla\cdot(\kappa_\perp \nabla) \approx \mathbf{D}_\perp := (\bm{D}_{qq}^{(\bm{K}_q)} + \bm{D}_{rr}^{(\bm{K}_r)}) + \bm{D}_r\bm{K}_{qr}\bm{D}_q + \bm{D}_q\bm{K}_{qr}^T\bm{D}_r.
\end{align}
Note that the \ac{sbp} property allows us to re-write the two right most terms in \eqref{eq:perpendicular FD operator} when multiplying by $\bm{H}$ such that,
\begin{align}\label{eq:curvilinear cross derivative sbp}
\begin{aligned}
    \bm{H}\bm{D}_q\bm{K}_{qr}\bm{D}_r &= - \bm{D}_q^T\bm{H}\bm{K}_{qr}\bm{D}_r + \bm{H}_r\bm{B}_q\bm{K}_{qr}\bm{D}_r, \\
    \bm{H}\bm{D}_r\bm{K}_{qr}\bm{D}_q &= - \bm{D}_r^T\bm{H}\bm{K}_{qr}\bm{D}_q + \bm{H}_q\bm{B}_r\bm{K}_{qr}\bm{D}_q.
\end{aligned}
\end{align}
This fact will be useful when proving stability.

\begin{remark}
    The coordinate transformations, $q_x$, $q_y$, $r_x$ and $r_y$, can be computed by \ac{sbp} finite differences.
\end{remark}

\subsubsection{Multiblock and boundary conditions}\label{sec:sub:sub:multiblock and boundary conditions}

\begin{figure}[H]
    \centering
    \includegraphics[width=\linewidth]{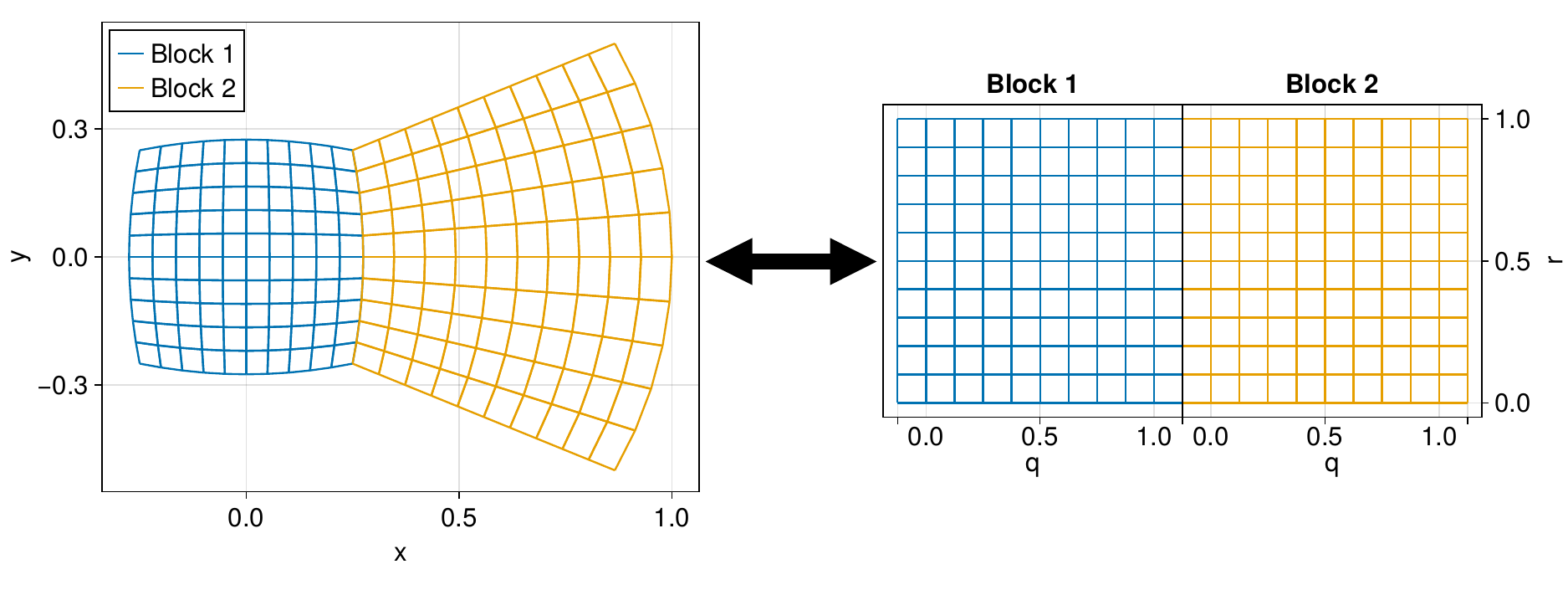}
    \caption{Diagram of a multi-block curvilinear domain with two sub-blocks. The physical sub-domains (left) are  mapped to reference sub-domains (right).}
    \label{fig:multiblock diagram}
\end{figure}
To perform the computations in complex geometry we decompose the domain into multiple sub-blocks. Therefore, consider splitting the real space domain into two regions $\Omega = \Omega^- \cup \Omega^+$ separated by a smooth boundary, for example as depicted in as illustrated in Figure \ref{fig:multiblock diagram}.

Each subdomain can be represented by the reference domain $\Omega_c$ by way of coordinate transform.
Note that, while we only use two blocks here, the ideas presented in this section can be extended to more blocks, with the same results applying.

Boundary and interface (for connecting the solution between subdomains) conditions are enforced weakly by adding \acp{sat} \cite{carpenter_time-stable_1994} to the perpendicular operator giving,
\begin{align}\label{eq:perpendicular diffusion SAT}
    \bm{P}_\perp \bm{u} := (I_2 \otimes \bm{D}_\perp) + \mathbf{SAT}_q + \mathbf{SAT}_r + \mathbf{SAT}_I
\end{align}
where $\bm{D}_\perp$ is as defined in equation \eqref{eq:perpendicular FD operator} such that the first term on the right applies $\bm{D}_\perp$ to either sub-block. The terms $\mathbf{SAT}_q$ and $\mathbf{SAT}_r$ enforce the Dirichlet boundary conditions and $\mathbf{SAT}_I$ enforces the continuity of the solution and its derivative across the block interface. Each \ac{sat} has penalty parameters associated with it, which are chosen to ensure the numerical scheme is energy stable by mimicking the energy estimate of the continuous problem

The interface term in curvilinear coordinates is,
\begin{align}
    &\begin{aligned}
    \mathbf{SAT}_{I,q} &= \tau_{I,0} \bm{H}^{-1} \hat{\bm{L}}_{q,0} \bm{u} + 
            \tau_{I,1} \bm{H}^{-1}_q \left( \hat{\bm{L}}_{q,1} \bm{K}_q \bm{D}_q + \hat{\bm{L}}_{q,1} \bm{K}_{qr} \bm{D}_r \right) \bm{u} + 
        \\&\qquad 
            \tau_{I,2} \bm{H}^{-1}_q \left( (\bm{K}_q \bm{D}_q)^T \hat{\bm{L}}_{q,1} +
                 (\bm{K}_{qr} \bm{D}_r)^T \hat{\bm{L}}_{q,1} \right) \bm{u}
    \end{aligned}\label{eq:SAT interface}
\end{align}
where 
\begin{align}\label{eq:Preliminaries:Interface L0 and L1}
    \hat{\bm{L}}_0 = I_q \otimes \begin{bmatrix}
        \bm{e}_n\bm{e}_n^T & -\bm{e}_n\bm{e}_1^T \\ - \bm{e}_{1}\bm{e}_n & \bm{e}_1\bm{e}_1^T
    \end{bmatrix}, \qquad
    \hat{L}_1 = I_q \otimes \begin{bmatrix}
        \bm{e}_n\bm{e}_n^T & -\bm{e}_n\bm{e}_1^T \\ \bm{e}_1\bm{e}_n^T & -\bm{e}_1\bm{e}_1^T
    \end{bmatrix},
\end{align}
where $\bm{e}_i$ is the $i$th unit vector, such that the above matrices have the elements,
\begin{align}
    [\hat{L}_0 \bm{u}]_i = \begin{cases}
        \bm{u}^-_n - \bm{u}^+_1 &\qquad \text{if } i=n,\\
        \bm{u}^+_1 - \bm{u}^-_n &\qquad \text{if } i=n+1,\\
        0                       &\qquad \text{otherwise},
    \end{cases} \qquad 
    [\hat{L}_1 \bm{u}]_i = \begin{cases}
        \bm{u}^-_n - \bm{u}^+_1 &\qquad \text{if } i=n,\\
        \bm{u}^-_n - \bm{u}^+_1 &\qquad \text{if } i=n+1,\\
        0                       &\qquad \text{otherwise}.
    \end{cases}
\end{align}
The $\hat{\bm{L}}_0$ and $\hat{\bm{L}}$ matrices select the elements of $\bm{u}$ corresponding to the block interfaces.

The \acp{sat} for the Dirichlet boundary conditions in curvilinear coordinates are,
\begin{align}\label{eq:Dirichlet SAT q}
    &\begin{aligned} \mathbf{SAT}_q &=
        \tau_{q,0} \bm{H}_q^{-1}(\bm{B}_{n_q} - \bm{B}_{1_q}) \bm{K}_q \bm{H}^{-1}_q (\bm{B}_{n_q} - \bm{B}_{1_q})(\bm{u}-\bm{g})
        + \\
        &\tau_{q,1} \bm{H}_q^{-1}\left( (\bm{K}_q \bm{H}_q \bm{D}_q)^T + 
            (\bm{K}_{qr}\bm{H}_q\bm{D}_r)^T \right)\bm{H}_q^{-1} (\bm{B}_{n_q} - \bm{B}_{1_q})(\bm{u}-\bm{g}),
    \end{aligned}
        \\
    &\begin{aligned}
    \mathbf{SAT}_r &= \tau_{r,0} \bm{H}_r^{-1}(\bm{B}_{n_r} - \bm{B}_{1_r}) \bm{K}_r \bm{H}_r^{-1} (\bm{B}_{n_r} - \bm{B}_{1_r})(\bm{u}-\bm{g})
        + \\
        &\qquad\tau_{r,1} \bm{H}_r^{-1} \left( (\bm{K}_r \bm{H}_r \bm{D}_r)^T + 
        (\bm{K}_{qr}\bm{H}_r\bm{D}_q)^T \right)\bm{H}_r^{-1} \bm{B}_r (\bm{u}-\bm{g}).
    \end{aligned}\label{eq:Dirichlet SAT r}
\end{align}
where $\bm{g}$ is a vector of the boundary data and where
\begin{align}
    \bm{B}_{i_q} = \bm{e}_i\bm{e}_i^T\otimes I_r,\quad \bm{B}_{j_r} = I_q\otimes \bm{e}_j\bm{e}_j^T, \quad i\in\{1,n_q\}, \, j\in\{1,n_r\},
\end{align}
pick out the elements along the exterior boundaries in the $q$ or $r$ directions.

\subsubsection{Semi-discrete Stability}

The semi-discrete form with boundary and interface conditions can then be written,
\begin{align}
    \dv{\bm{u}}{t} = (\bm{P}_\perp + \bm{P}_\parallel) \bm{u},
\end{align}
where $\bm{P}_\parallel$ is discussed in \S\ref{sec:sub:numerical parallel operator}. Since there is no change to the parallel operator from the previous work which would change the analysis \cite{muir_provably_2025}, therefore the results of \cite{muir_provably_2025} can be used to show that the negative semi-definiteness of the parallel operator hold. Hence, we continue the discussion focusing on the perpendicular operator.

Analysis from here-on out is performed in $(q,r)$-logical coordinates. We now introduce a number of matrices which will assist with the analysis by splitting $\bm{P}_\perp$ into boundary and interior components. First consider the matrices,
\begin{align}
    B_{qI} = \begin{bmatrix} \bm{e}_{n_q}\bm{e}_{n_q}^T & \bm{0}_{n_q} \\ \bm{0}_{n_q} & -\bm{e}_{1_q}\bm{e}_{1_q}^T \end{bmatrix}, \qquad 
    B_{qB} = \begin{bmatrix} \bm{e}_{1_q}\bm{e}_{1_q}^T & \bm{0}_{n_q} \\ \bm{0}_{n_q} & -\bm{e}_{n_q}\bm{e}_{n_q}^T \end{bmatrix}, \qquad
    B_{rB} = \begin{bmatrix} \bm{e}_{n_q}\bm{e}_{n_q}^T & \bm{0}_{n_q} \\ \bm{0}_{n_q} & -\bm{e}_{1_q}\bm{e}_{1_q}^T \end{bmatrix}.
\end{align}
where $\bm{e}_{1_q}$ and $\bm{e}_{n_q}$ are elementary unit vectors of length $n_q$ with a 1 in the first and $n_q$th position and similarly $\bm{e}_{1_r}$ and $\bm{e}_{n_r}$.
From these we define,
\begin{align}
    \begin{array}{lll}
        \bm{B}_{qI} = B_{qI}\otimes I_{n_r}, &\qquad
            \bm{B}_{qB} = B_{qB}\otimes I_{n_r}, &\qquad
            \bm{B}_{rB} = I_2 \otimes I_{n_q} \otimes B_y    \\
        \bm{H}_{qB} = \bm{H}_q(\bm{B}_{qB}\bm{B}_{qB}^T), &\qquad 
            \bm{H}_{qI} = \bm{H}_q(\bm{B}_{qI}\bm{B}_{qI}^T), &\qquad
            \bm{H}_{rB} = \bm{H}_r(\bm{B}_{rB}\bm{B}_{rB}^T), \\
        \hat{\bm{H}}_q = \bm{H}_q - \bm{H}_{qB} - \bm{H}_{qI}, &\qquad
            \hat{\bm{H}}_r = \bm{H}_r - \bm{H}_{rB}, &\qquad
            \hat{\bm{H}} = \bm{H} - \bm{H}_{qI} - \bm{H}_{qB} - \bm{H}_{rB},
    \end{array}
\end{align}
such that $\bm{B}_{qI}$ extracts the values along the interface, $\bm{B}_{qB}$ the values along the $q$ boundaries and $\bm{B}_{rB}$ the $r$ boundary values.

To split $\bm{P}_\perp$ into components we will write,
\begin{align}
    \bm{H}(\bm{D}_\perp)_{interior} = -\bm{D}_q^T\bm{K}_q\hat{\bm{H}}_q\bm{D}_q - \bm{D}_r^T\bm{K}_r\hat{\bm{H}}_r\bm{D}_r - \bm{D}_q^T\hat{\bm{H}}\bm{K}_{qr}\bm{D}_r - \bm{D}_r^T\hat{\bm{H}}\bm{K}_{qr}\bm{D}_q - \bm{R}_q - \bm{R}_r.
\end{align}
Hence,
\begin{align}
    (\bm{P}_\perp)_{interior} = I_2 \otimes (\bm{D}_\perp)_{interior}.
\end{align}
The boundary terms and \acp{sat} can be collected to give,
\begin{align}\label{eq:perpendicular boundary operator}
    \begin{aligned}
    \bm{H}(\bm{P}_\perp)_B &= -\bm{D}_q^T\bm{K}_q\bm{H}\bm{D}_q - \bm{D}_r^T\bm{K}_r\hat{\bm{H}}\bm{D}_r - \bm{D}_q^T\hat{\bm{H}}\bm{K}_{qr}\bm{D}_r - \bm{D}_r^T\hat{\bm{H}}\bm{K}_{qr}\bm{D}_q + 
    \\&\qquad
        \bm{B}_q\bm{K}_q\bm{D}_q + \bm{B}_r\bm{K}_r\bm{D}_r + \bm{H}_r\bm{B}_q\bm{K}_{qr}\bm{D}_r + \bm{H}_q\bm{B}_r\bm{K}_{qr}\bm{D}_q +
    \\&\qquad
        \bm{H}\mathbf{SAT}_q + \bm{H}\mathbf{SAT}_r + \bm{H}\mathbf{SAT}_I.
    \end{aligned}
\end{align}

From here we note that $(\bm{P}_\perp)_{interior}$ contains only symmetric terms, and so is symmetric and negative definite by construction. It is therefore left for us to prove that $(\bm{P}_\perp)_B$ is symmetric and negative semi-definite.

\begin{theorem}\label{thrm:perpendicular energy bound}
    Consider the boundary operator \eqref{eq:perpendicular boundary operator} with Dirichlet \acp{sat} defined in \eqref{eq:Dirichlet SAT q} and \eqref{eq:Dirichlet SAT r}. Let $\bm{H}=\bm{H}_q\otimes\bm{H}_r$ where $\bm{H}_q$ and $\bm{H}_r$ are as defined in \eqref{eq:Curvilinear H} with $h_j>0$. If $\bm{g}_q=0$ and $\bm{g}_r=0$, and we have the penalties 
    \begin{align}
        \tau_{q,1}=-1, \quad \tau_{q,0}\ge\bm{K}_q, \qquad
        \tau_{r,1}=-1, \quad \tau_{r,0}\ge\bm{K}_r
    \end{align}
    then,
    \begin{align}
        \bm{P}_\perp = -\bm{H}^{-1}\bm{A}_\perp, \quad \bm{A}_\perp=\bm{A}_\perp^T, \quad \bm{v}^T\bm{A}_\perp\bm{v}\geq0, \quad \forall \bm{v}\in\bm{R}^{n_q\times n_r}.
    \end{align}
\end{theorem}

The proof can be found in Appendix \ref{appendix:Perpendicular operator SPD}.
To facilitate the proof we introduce the matrices,
\begin{align}\label{eq:SAT q Dirichlet matrix}
    \mathcal{A}_q = \begin{bmatrix}
        \tau_{q,0}         & [\bm{K}_q]_{j,k}    & [\bm{K}_{qr}]_{j,k}    \\
        [\bm{K}_q]_{j,k}    & [\bm{K}_q]_{j,k}    & [\bm{K}_{qr}]_{j,k}    \\
        [\bm{K}_{qr}]_{j,k} & [\bm{K}_{qr}]_{j,k} & [\bm{K}_r]_{j,k}
    \end{bmatrix},\quad 
    \mathcal{A}_r = \begin{bmatrix}
        \tau_{r,0}         & [\bm{K}_r]_{j,k}    & [\bm{K}_{qr}]_{j,k}    \\
        [\bm{K}_r]_{j,k}    & [\bm{K}_r]_{j,k}    & [\bm{K}_{qr}]_{j,k}    \\
        [\bm{K}_{qr}]_{j,k} & [\bm{K}_{qr}]_{j,k} & [\bm{K}_q]_{j,k}
    \end{bmatrix}
\end{align}
and
\begin{align}\label{eq:SAT interface matrix}
    \mathcal{G} = \begin{bmatrix}
        \tau_{I,0} 					& \half [\bm{K}_q^-]_{n_q,i}			
            & \half [\bm{K}_q^+]_{1,i}	& \half [\bm{K}_{qr}^-]_{n_q,i} 	
            & \half [\bm{K}_{qr}^+]_{1,i} \\
        \half [\bm{K}_q^-]_{n_q,i}  	& [\bm{H}_q\bm{K}_q^-]_{n_q,i}	
            & 0                         & [\bm{H}_q\bm{K}_{qr}^-]_{n_q,i}
            & 0 \\
        \half [\bm{K}_q^+]_{1,i}    	& 0
            & [\bm{H}_q\bm{K}_q^+]_{1,i}& 0
            & [\bm{H}_q\bm{K}_{qr}^+]_{1,i} \\
        \half [\bm{K}_{qr}^-]_{n_q,i}   & [\bm{H}_q\bm{K}_{qr}^-]_{n_q,i}
            & 0                         & [\bm{H}_q\bm{K}_r^-]_{n_q,i} 	
            & 0 \\
        \half [\bm{K}_{qr}^+]_{1,i} 	& 0
            & [\bm{H}_q\bm{K}_{qr}^+]_{1,i} & 0
            & [\bm{H}_q\bm{K}_r^+]_{1,i}
    \end{bmatrix}
\end{align}

The matrix $\mathcal{A}_q$ collects all relevant terms along the $q$ Dirichlet boundaries at a particular $r$ value and $\mathcal{A}_r$ is defined similarly for the $r$ boundaries. The second matrix $\mathcal{G}$ collects all terms along the multi-block interface.

For $\mathcal{A}_q$ (and therefore $\mathcal{A}_r$) and $\mathcal{G}$ the following lemmas hold,

\begin{lemma}\label{lem:SAT q Dirichlet matrix}
    Consider the matrix $\mathcal{A}_q$ given by \eqref{eq:SAT q Dirichlet matrix}. If $\tau_{0,q}\geq [\bm{K}_q]_{j,k}$, then $\bm{v}^T\mathcal{A}_q\bm{v}\geq 0$, $\forall \bm{v}\in\R^3$.
\end{lemma}

\begin{lemma}\label{lem:SAT r Dirichlet matrix}
    Consider the matrix $\mathcal{A}_r$ given by \eqref{eq:SAT q Dirichlet matrix}. If $\tau_{0,r}\geq [\bm{K}_r]_{j,k}$, then $\bm{v}^T\mathcal{A}_r\bm{v}\geq 0$, $\forall \bm{v}\in\R^3$.
\end{lemma}

\begin{lemma}\label{lem:SAT interface matrix}
    Consider the matrix $\mathcal{G}$ given in equation \eqref{eq:SAT interface matrix}. If $\tau_{I,0} \geq \frac{1}{4} ([\bm{K}_q^-]_{n_q,i}/[\bm{H}_q]_{n_q,i}, [\bm{K}_q^+]_{1,i}/[\bm{H}_q]_{1,i})$ then $\bm{v}^T\mathcal{G}\bm{v}$, $\forall \bm{v}\in\R^5$.
\end{lemma}

The proofs for lemmas \ref{lem:SAT q Dirichlet matrix} and \ref{lem:SAT interface matrix} can also be found in Appendix \ref{appendix:Perpendicular operator SPD}, we do not prove lemma \ref{lem:SAT r Dirichlet matrix} since the proof is identical to \ref{lem:SAT q Dirichlet matrix}.

\section{Fully-discrete analysis}\label{sec:Discrete analysis}

Details for the fully discrete form are found in \cite{muir_provably_2024-1}. However, we briefly describe it here for completeness.

Let $\Delta t_l = t_{l+1} - t_l$ be the time step size where $0\leq t_l < t_{l+1} \leq T$ such that the discrete solution at time $t=t_l$ is denoted $\bm{u}^l$.
Operator splitting is used to separate the slow (perpendicular) and fast (parallel) time scales. For a single time step the perpendicular diffusion is solved first by $\theta$-method where for $\theta=\half$ we have,
\begin{align}
    \left(\bm{I} - \half\Delta t\bm{P}_\perp\right)\bm{u}^{l+\half} = \left( \bm{I} + \half\Delta t\bm{P}_\perp \right)\bm{u}^l + \half\Delta t(\bm{F}^{l} + \bm{F}^{l+\half}).
\end{align}
where $\bm{F}^l = \begin{bmatrix}f(q_1, r_1, t^l), f(q_2, r_1, t^l), \dots, f(q_n, r_n, t^l)\end{bmatrix}^T$ is a forcing term.
The parallel diffusion is then advanced by implicit midpoint rule,
\begin{align}
    \bm{w}_f^{l+\half} &= \bm{P}_f \bm{u}^{l+\half}, \qquad 
    \bm{w}_b^{l+\half} = \bm{P}_f \bm{u}^{l+\half}, \\
    \bm{u}^{l+1} &= \bm{u}^{l+\half} - \tau_\parallel \kappa_\parallel \bm{H}^{-1} \left( \bm{u}^{l+1} - \half\left[\bm{w}^{l+\half}_f + \bm{w}^{l+\half}_b\right] \right).
\end{align}

The following theorem states the energy for the fully discrete form.
\begin{theorem}\label{thrm:Stability of discrete problem}
   Consider The fully discrete form of \eqref{eq:ADE cont laplace form} with $\mathbf{P}_\perp$ given by \eqref{eq:perpendicular diffusion SAT} and $\mathbf{P}_\parallel$ given by \eqref{eq:Numerical parallel opeartor}. Then,
   \begin{align}
       \bm{u}^T(\bm{H}\bm{P}_\perp + (\bm{H}\bm{P}_\perp)^T)\bm{u} \leq 0 \qquad \bm{u}^T(\bm{H}\bm{P}_\parallel + (\bm{H}\bm{P}_\parallel)^T)\bm{u} \leq 0 \qquad \forall\bm{u}\in\R^{n_x n_y}.
   \end{align}
\end{theorem}

For proof we refer to Theorem 3.5 in ~\cite{muir_provably_2025}. One can replace the operators with their variable coefficient forms and see the result holds if one includes the proof for theorem \eqref{thrm:perpendicular energy bound}.
In practice an interpolating polynomial is constructed in each domain using cubic Hermite splines from the \textit{CubicHermiteSpline.jl} Julia package \cite{noauthor_liuyxppcubichermitesplinejl_nodate}.

\section{Code verification and numerical experiment}\label{sec:Verification}

In this section we perform numerical experiments to verify the accuracy and stability of the method.
We first test by manufactured solution on the perpendicular problem only, with no parallel map. We then introduce the parallel map though a magnetic field with a single magnetic island. All curvilinear tests are performed on a circular domain, which is introduced in \S\ref{sec:sub:Convergence in multiblock curvilinear geometry}.

In all cases, errors are determined using the norm defined by the matrix $H$,
\begin{align}\label{eq:relative error}
    \text{err}_{\text{rel}} = \frac{\|\bm{u} - \bm{u}^\star\|_{\bm{H}}}{\|\bm{u}^\star\|_{\bm{H}}},
\end{align}
where $\bm{u}^\star$ is the ``exact solution'' and $\|\cdot\|_{\bm{H}}$ is defined in \eqref{eq:H inner product norm}.

Finally, we perform an experiment on a bean-shaped domain using the Stepped Pressure Equilibrium Code (SPEC) to generate the magnetic geometry.

\subsection{Convergence in multiblock curvilinear geometry}\label{sec:sub:Convergence in multiblock curvilinear geometry}

We now perform convergence tests in the multi-block and curvilinear geometry. Consider a circular domain with outer boundary defined by, $(x(\theta),y(\theta)) = (\cos(\theta),\sin(\theta))$.
To apply the \ac{sbp} finite difference discretisation, the grid is broken into 5 subdomains as shown in Figure \ref{fig:Circle dilation effect}.

\begin{figure}[H]
    \centering
    \includegraphics[width=0.7\linewidth]{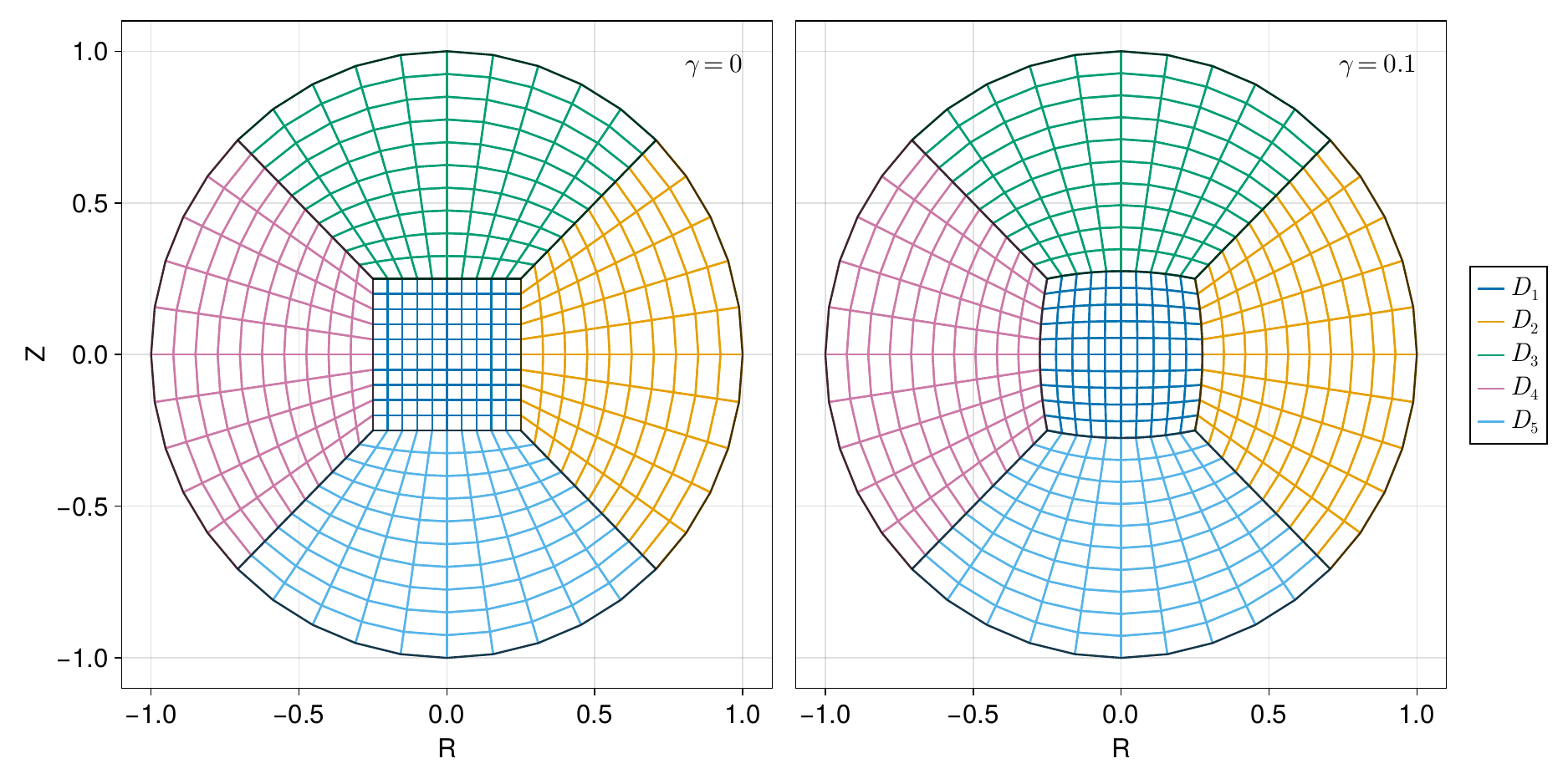}
    \caption{Effect of the dilation parameter on the mesh geometry. The value for the dilation parameter is listed to the top right of the subplot, with $\gamma=0$ being no dilation and $\gamma=0.1$ having a slight dilation.}
    \label{fig:Circle dilation effect}
\end{figure}

To add to the investigations conducted in \cite{muir_efficient_2023}  we include a term in the parametrisation of the interior domain to dilate the boundaries. Having a field aligned grid can reduce pollution, therefore we wish to investigate if dilating the interior (to align it more to field lines) improves convergence or reduces errors. The boundary functions for constructing the interior domain are,
\begin{align}\label{eq:Transfinite interpolation square}
        \begin{aligned}
        c_1(u) &= [-0.25,-0.25]^T + u([0.25,-0.25]^T - [-0.25,-0.25]^T) + u(1-u)[0.0,-\gamma]^T, \\
        c_2(v) &= [0.25,-0.25]^T + v([0.25,0.25]^T - [0.25,-0.25]^T) + v(1-v)[\gamma,0.0]^T, \\
        c_4(u) &= [-0.25,-0.25]^T + v([-0.25,0.25]^T - [-0.25,-0.25]^T) + v(1-v)[-\gamma,0.0]^T, \\
        c_3(v) &= [-0.25,0.25]^T + u([0.25,0.25]^T - [-0.25,0.25]^T) + u(1-u)[0.0,\gamma]^T, \\
    \end{aligned}
\end{align}
where $\gamma$ dilates the boundaries of the domain as shown in Figure \ref{fig:Circle dilation effect}.

\subsubsection{Manufactured solution}

To test the perpendicular operator only, no magnetic field is applied, so that $\mathcal{P}_\parallel=0$. We choose the analytic solution,
\begin{align}
    \tilde{u}(x,y,t) = \cos(2\pi\omega_t t) \sin(2\pi\omega_R R)\sin(2\pi\omega_Z Z),
\end{align}
with parameters shown in table \ref{tab:MMS parameters}.
The source term and initial condition are given by,
\begin{align}
    F(x,y,t) = \pdv{\tilde{u}}{t} - \nabla\cdot(\kappa_\perp \nabla\tilde{u}), \qquad u(x,y,0) = \tilde{u}(x,y,0).
\end{align}
On the outer boundary we apply the Dirichlet boundary condition $u(x,y,t) = \tilde{u}(x_b,y_b,t)$ where $x_b^2+y_b^2=1$. Continuity of the solution and derivatives across domain boundaries are enforced by interface \acp{sat} as detailed in \S\ref{sec:sub:sub:multiblock and boundary conditions}. We run convergence tests on the non-dilated and dilated circles, with dilation parameter $\gamma$ specified in the legend.

\begin{table}[H]
    \centering
    \caption{Parameters for convergence tests in Figure \ref{fig:MMS circle}.}
    \begin{tabular}{cccc}\hline
         Test        & $\omega_x$   & $\omega_y$    & $\omega_t$ \\\hline
         Spatial     & $5.5$        & $7$           & $3$ \\
    \end{tabular}
    \label{tab:MMS parameters}
\end{table}

Results for the method of manufactured solutions \cite{roache_code_2001} are shown in Figure \ref{fig:MMS circle}. Both second and fourth order methods achieve optimal convergence, though we note that the fourth order requires higher resolution before convergence begins. This is due to the relatively wide boundary stencil for the fourth order requiring 6 grid points per-side.

\begin{figure}[H]
    \centering
    \begin{tikzpicture}
    \begin{groupplot}[
            group style={
            group size=2 by 1,
            x descriptions at=edge bottom,
            vertical sep    =10pt,
            horizontal sep  =10pt
        },
        %
        enlargelimits=true,
        height=0.4\textwidth,
    ]
    \nextgroupplot[
        legend entries={
                $\gamma = 0$,
                $\gamma = 0.1$,
            },
        legend style={
                at={([yshift=4pt]1.0,1.1)},
                anchor=north
            },
        legend columns=3,
        ymode=log,xmode=log,
        xtick={11,21,31,41,51,61,71},
        xticklabels={11,21,31,41,51,61},
        xlabel=Subdomain $n_q\equal n_r$,
        ylabel=log(relative error),
        grid
        ]

        \addplot table[x index=0,y index=1,col sep=comma,skip first n=1]{./data/MMS_Circle.csv};
        \addplot table[x index=0,y index=2,col sep=comma,skip first n=1]{./data/MMS_Circle.csv};

        \node[] at (rel axis cs: 0.75,0.9) {\large 2nd order};

        \logLogSlopeTriangle{0.6}{0.5}{0.89}{2.5}{blue,dashed};
        
    \nextgroupplot[
        ymode=log,xmode=log,
        yticklabel pos=right,
        xtick={11,21,31,41,51,61,71},
        xticklabels={11,21,31,41,51,61},
        xlabel=Subdomain $n_q\equal n_r$,
        ylabel=log(relative error),
        grid
        ]

        \addplot table[x index=0,y index=3,col sep=comma,skip first n=1]{./data/MMS_Circle.csv};
        \addplot table[x index=0,y index=4,col sep=comma,skip first n=1]{./data/MMS_Circle.csv};

        \node[] at (rel axis cs: 0.75,0.9) {\large 4th order};

        \logLogSlopeTriangle{0.7}{0.4}{0.85}{4}{blue,dashed};

    \end{groupplot}
    \end{tikzpicture}
    \caption{Convergence for the perpendicular only manufactured solution on the circular domains with dilation of the interior domain given by $\gamma$. The effect of $\gamma$ is illustrated by Figure \ref{fig:Circle dilation effect}.
    \textbf{Left}: Second order.
    \textbf{Right}: Fourth order.}
    \label{fig:MMS circle}
\end{figure}
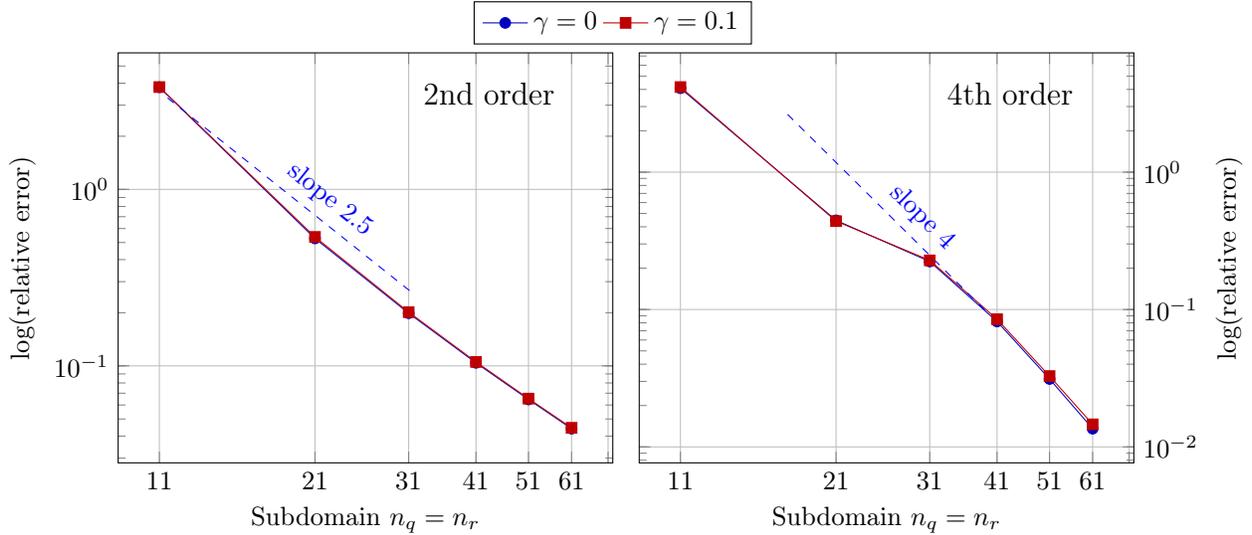

\subsubsection{Self convergence with island}

We now perform a self convergence test on the circular domain with a parallel map, by comparing lower resolution runs to a high resolution reference solution. We choose the magnetic field presented in \cite{chacon_asymptotic-preserving_2014},
\begin{align}\label{eq:Circle self convergence magnetic field}
    \psi(r,\theta) = (r - r_1)^2 + \delta (r - 0.5) (1-r) \cos(\theta), \qquad B=\grad\psi\times \bm{z}
\end{align}
for which we choose $\delta = 0.05$ and $r_1=0.7$. The computational grid is in $(x,y)$ coordinates, so to construct the parallel map we first map the grid to the field coordinates $(r,\theta)$, perform the field line tracing, then map the resulting point back to $(x,y)$. 
The initial condition and boundary conditions are,
\begin{align}
    u(x,y,0) = 0, \qquad\text{and}\qquad   u(r,\theta,t) = 0, \quad r = 1.
\end{align}
Furthermore, the source term,
\begin{align}\label{eq:Single island self convergence source term}
    Q(r,t) = 4(1-r^2)^8,
\end{align}
is added to avoid the trivial solution. The reference solution is computed using $n_q\times n_r=101\times 101$ grid points in each domain. 

In Figure \ref{fig:circle self convergence} we show the results for convergence tests with $\kappa_\parallel=10^6$ and $10^9$ in a case with the boundary dilation set to $\gamma=0.1$. Time steps for each run are scaled with the grid resolution so that $\Delta t = 10^{-2} / (n_r - 1)^2$.

Grid packing is used to increase the resolution of the solution near the separatrix, with packing function
\begin{align}
    u(s) = r_s + \alpha \sinh( \arcsinh\left(\frac{1 - r_s}{\alpha}\right)s + \arcsinh\left(\frac{-r_s}{\alpha}\right)(1-s) )
\end{align}
where we choose $\alpha = 0.1$ and here $u$ is the variable input to the transfinite interpolation functions (for example equation \eqref{eq:Transfinite interpolation square} for the interior domain). The parameter $\alpha$ controls the tightness of the packing, we note that smaller values are valid, but $\approx0.05$ will fail in the parallel mapping step, likely due to the triangulation required in the cubic Hermite interpolation.

\begin{figure}[H]
    \centering
    \begin{tikzpicture}
    \begin{groupplot}[
            group style={
            group size=2 by 1,
            x descriptions at=edge bottom,
            vertical sep    =10pt,
            horizontal sep  =10pt
        },
        %
        enlargelimits=true,
        height=0.4\textwidth,
        cycle multiindex* list={
            {blue!75!black,     mark=o, mark size=3 pt, solid},
            {red!75!black,      mark=o, mark size=3 pt, solid},
            {blue!75!black,     mark=x, mark size=5 pt, dashed},
            {red!75!black,      mark=x, mark size=5 pt, dashed},
        }
    ]
    \nextgroupplot[
        legend entries={
                $\kappa_\parallel = 10^6$,
                $\kappa_\parallel = 10^9$,
            },
        legend style={
                at={([yshift=4pt]1.0,1.1)},
                anchor=north,
                /tikz/every even column/.append style={column sep=0.5cm}
            },
        legend columns=2,
        ymode=log,xmode=log,
        xtick={105,155,205,255,305,355,405},
        xticklabels={21,31,41,51,61,71,81},
        xlabel=Subdomain $n_q\equal n_r$,
        ylabel=log(relative error),
        grid
        ]

        \addplot table[x index=0,y index=1,col sep=comma,skip first n=1]{./data/SISC_O2_EXP6_gamma0.0_g0b0.25_alpha0.1.csv};
        \addplot table[x index=0,y index=1,col sep=comma,skip first n=1]{./data/SISC_O2_EXP9_gamma0.0_g0b0.25_alpha0.1.csv};

        \node[] at (rel axis cs: 0.75,0.9) {\large 2nd order};

        \logLogSlopeTriangle{0.6}{0.4}{0.91}{1}{blue,dashed};
        \logLogSlopeTriangle{0.93}{0.3}{0.6}{2}{blue,dashed};
        
    \nextgroupplot[
        ymode=log,xmode=log,
        yticklabel pos=right,
        xtick={105,155,205,255,305,355,405},
        xticklabels={21,31,41,51,61,71,81},
        xlabel=Subdomain $n_q\equal n_r$,
        ylabel=log(relative error),
        grid
        ]

        \addplot table[x index=0,y index=1,col sep=comma,skip first n=1]{./data/SISC_O4_EXP6_gamma0.0_g0b0.25_alpha0.1.csv};
        \addplot table[x index=0,y index=1,col sep=comma,skip first n=1]{./data/SISC_O4_EXP9_gamma0.0_g0b0.25_alpha0.1.csv};

        \node[] at (rel axis cs: 0.75,0.9) {\large 4th order};

        \logLogSlopeTriangle{0.7}{0.2}{0.8}{2}{blue,dashed};
        \logLogSlopeTriangle{0.96}{0.15}{0.55}{4}{blue,dashed};

    \end{groupplot}
    \end{tikzpicture}
    \caption{Self convergence test in the five block circle test with magnetic field given by \eqref{eq:Circle self convergence magnetic field} and no boundary dilation (i.e. $\gamma=0$). The fourth order method outperforms the second order case in terms of relative errors, improving them by an order of magnitude.
    The average convergence rates of the second order cases for each $\kappa_\parallel$ are $\sim1.2$ and $\sim 1$. 
    For the fourth order cases the average rates are $\sim1.7$ and $\sim1.8$. Note that in the fourth order cases the rate of convergence accelerates in the final stages to $\sim4$th order. Given the lower error and faster convergence rates we see that higher order is preferable.}
    \label{fig:circle self convergence}
\end{figure}
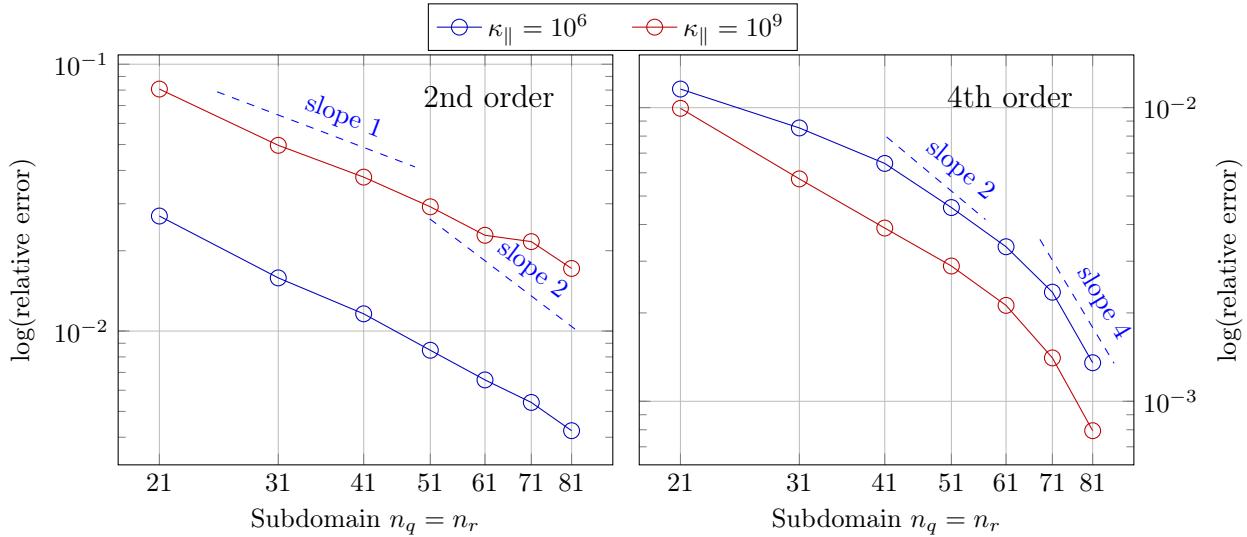

Convergence rates begin at first order and accelerate as the grid resolution is increased in both the second and fourth order cases. Furthermore, the errors are an order of magnitude lower in the fourth order case, which also increases in convergence rate more-so than the second order, achieving fourth order convergence at the final stages. This illustrates the need for higher order methods when solving the anisotropic diffusion equation in curvilinear geometries.
We note that only the convergence of the $\gamma=0$ case (no dilation) is shown, since the $\gamma=0.1$ case had higher errors. This suggests that deforming the interior grid to more closely align with the field lines may not be beneficial for reducing errors.

In Figure \ref{fig:Circular single island line plot} we show the solution along chords from $-0.9\leq x \leq -0.5$ at $y=0$ (the thickest part of the island), for varying grid resolutions and diffusion coefficients. In both cases it can be seen that the $51\times51$ resolution case performs qualitatively well compared to the $81\times81$. However, the solution near the separatrix in the $\kappa_\parallel=10^9$ case is still rough due to the rate at which flattening occurs.

\begin{figure}[H]
    \centering
    \includegraphics[width=0.7\linewidth]{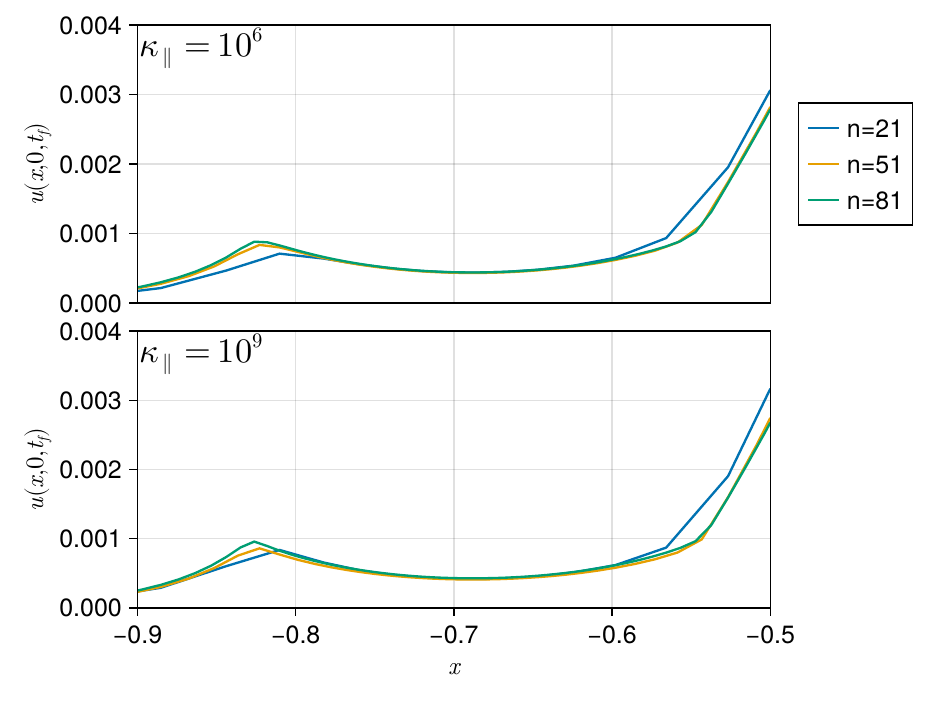}
    \caption{
    Solution along chords extending from $-0.9\leq x \leq -0.5$, $y=0$ for varying resolutions for the second order case. The legend indicates the grid resolution such that $n=n_x=n_y$.
    \textbf{Top}: $\kappa_\parallel=10^6$.
    \textbf{Bottom}: $\kappa_\parallel=10^9$.}
    \label{fig:Circular single island line plot}
\end{figure}

We plot the reference solution with a partial Poincar\'e section in the top of Figure \ref{fig:circle self convergence error field}. The solution is symmetric about the $x$-axis and so only the upper half plane is shown. The solution is cut off at a values $u>0.002$ in order to emphasise the island shape. In the lower half of Figure \ref{fig:circle self convergence error field} we show the point-wise relative error with $\bm{u}^\star = \bm{u}_{ref}$.
This shows that the errors are maximised at the boundary between the interior square domain and outer grids. This is likely because the grid is not aligned with magnetic field, and the coordinate Jacobian has a jump across the grid boundaries.

\begin{figure}[H]
    \centering
    \includegraphics[width=0.8\linewidth]{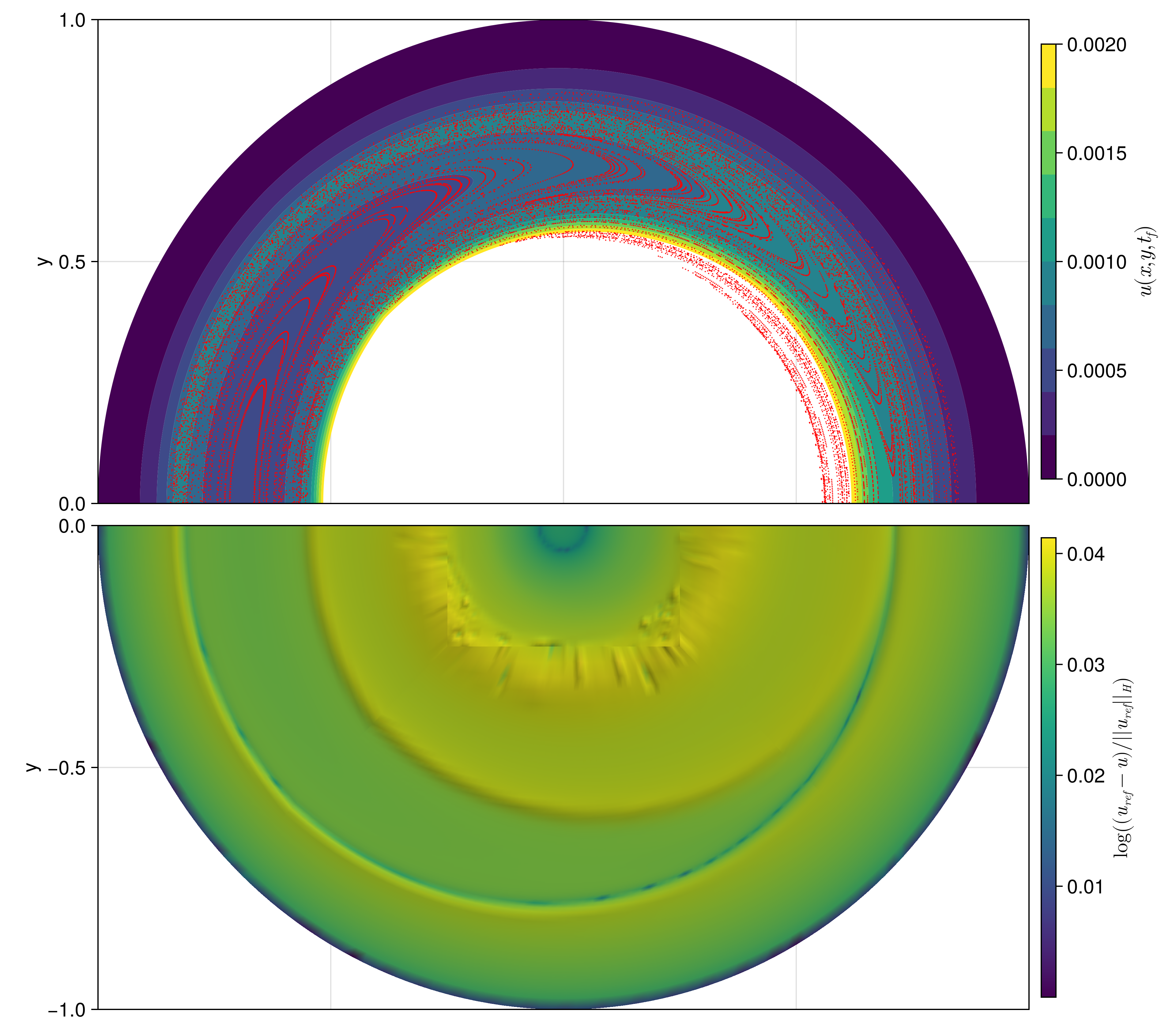}
    \caption{\textbf{Top}: Reference solution in upper half plane with partial Poincar\'e section overlayed to illustrate the solution conforming to the island shape.
    \textbf{Bottom}: The point-wise relative error. This can be seen to be maximised at the barrier between the interior grid and the outer domains, likely because the grid is not aligned with the magnetic field. The separatrix can also be seen in the error field as two curves of higher error meeting near $x=0.7$.}
    \label{fig:circle self convergence error field}
\end{figure}

\subsection{Solution in SPEC equilibria}\label{sec:Numerical experiments}

For an experimental test we use a magnetic equilibrium from SPEC. These equilibria are defined by a disjoint union of magnetic fields which occupy consecutive nested annular regions with a solid torus at the centre. The boundaries of each region are flux surfaces upon which $\bm{B}\cdot\bm{n}=0$, i.e. total transport barriers to magnetic field lines. For further detail we refer readers to references \cite{hudson_computation_2012,qu_stepped_2020}.
For this experiment we choose a single-volume, five field period stellarator for which the input file can be found in the SPEC repository\footnote{Input file: \lstinline{G3V01L0Fi.002.sp} from \href{https://github.com/PrincetonUniversity/SPEC}{SPEC Github}}. The parameterisation of SPEC annulus boundaries is included in the github repository to enable reproduction of figures in this paper \footnote{\href{https://github.com/Spiffmeister/FaADE_paperb}{https://github.com/Spiffmeister/FaADE\_paperb}}.
SPEC equilibria can be read with the Python or MATLAB routines provided with SPEC, or with the Julia package \textit{SPECReader}\footnote{\href{https://github.com/JuliaPlasma/SPECReader.jl}{https://github.com/JuliaPlasma/SPECReader.jl}}.

The computational plane is shown with Poincar\'e section and a low resolution version of the computational grid in Figure \ref{fig:SPEC equilibrium poincare}. The island chain is apparent at the edge, with widest islands around $Z\approx\pm0.55$, $R\approx5.7$. 

\begin{figure}[H]
    \centering
    \includegraphics[width=0.5\textwidth]{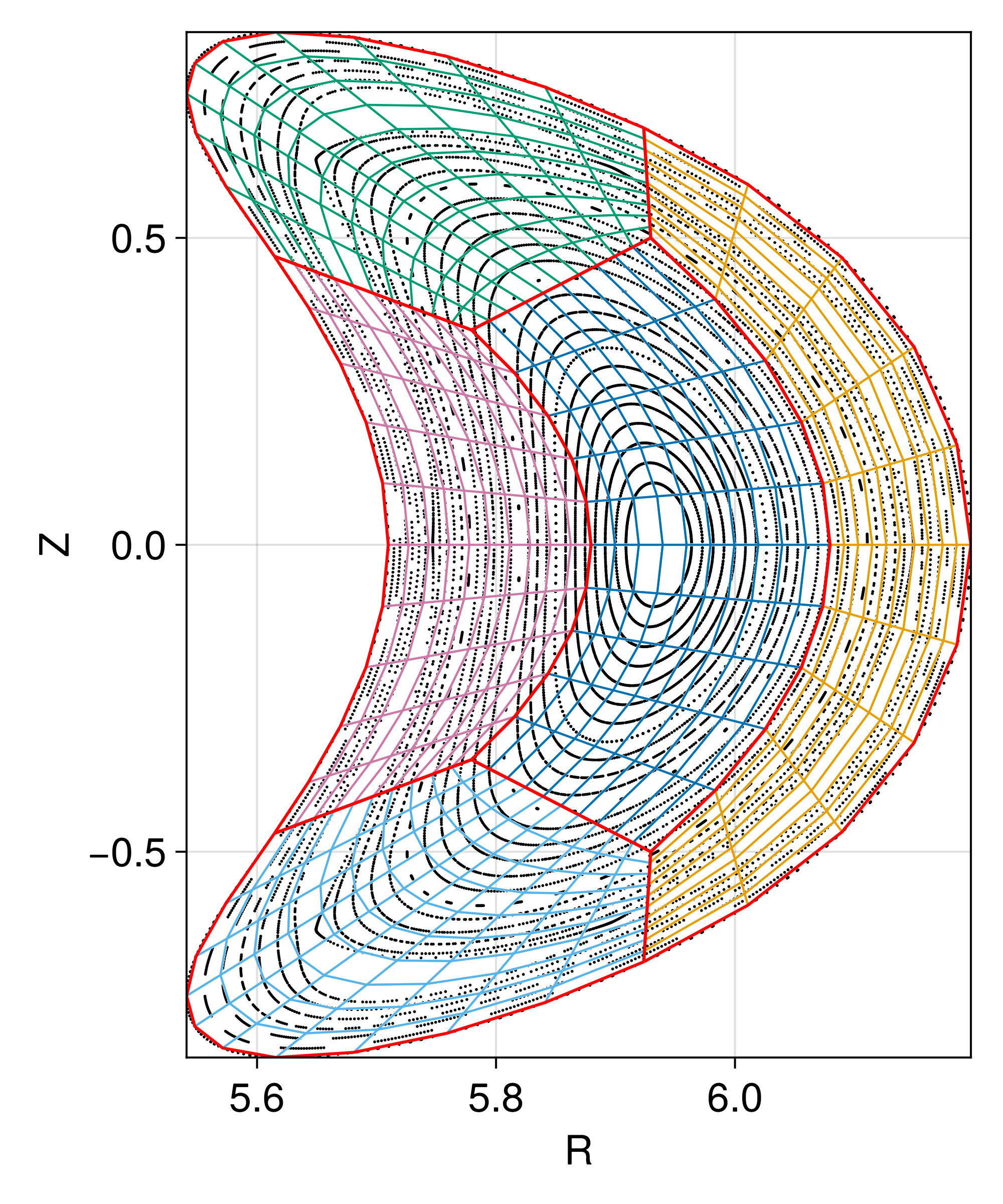}
    \caption{SPEC Poincar\'e section and low resolution multiblock grid. Note that the grid resolution here is purely for illustration and is not what is used in the computation. The red lines mark the subgrid boundaries.}
    \label{fig:SPEC equilibrium poincare}
\end{figure}

To avoid the trivial solution, a source term is added in the logical coordinates of SPEC which is equivalent to the single island case \eqref{eq:Single island self convergence source term},
\begin{align}
    u(s,\theta,t) = 4(1-s^2)^8, \qquad s\in[0,1]
\end{align}
The initial condition is zero everywhere $u(s,\theta,0) = 0$.

In Figure \ref{fig:SPEC equilibrium contour} we show the results of a simulation using the grid boundaries shown in Figure \ref{fig:SPEC equilibrium poincare} with interior domain having $n_q\times n_r = 61\times61$ grid points, and all others using $n_q\times n_r = 61\times 71$. Note that in the exterior regions the direction of the logical coordinate $r$ points from the interior region to the outer boundary, and does not correspond with the $Z$ direction. The time step and final time is $\Delta t = 10^{-8}$ and $t_f=5\times10^{-3}$ respectively. We also ran a case with $\Delta t = 10^{-7}$, but the results are similar, with the case presented here having a slightly flatter profile across the islands.

\begin{figure}[H]
    \centering
    \includegraphics[width=\linewidth]{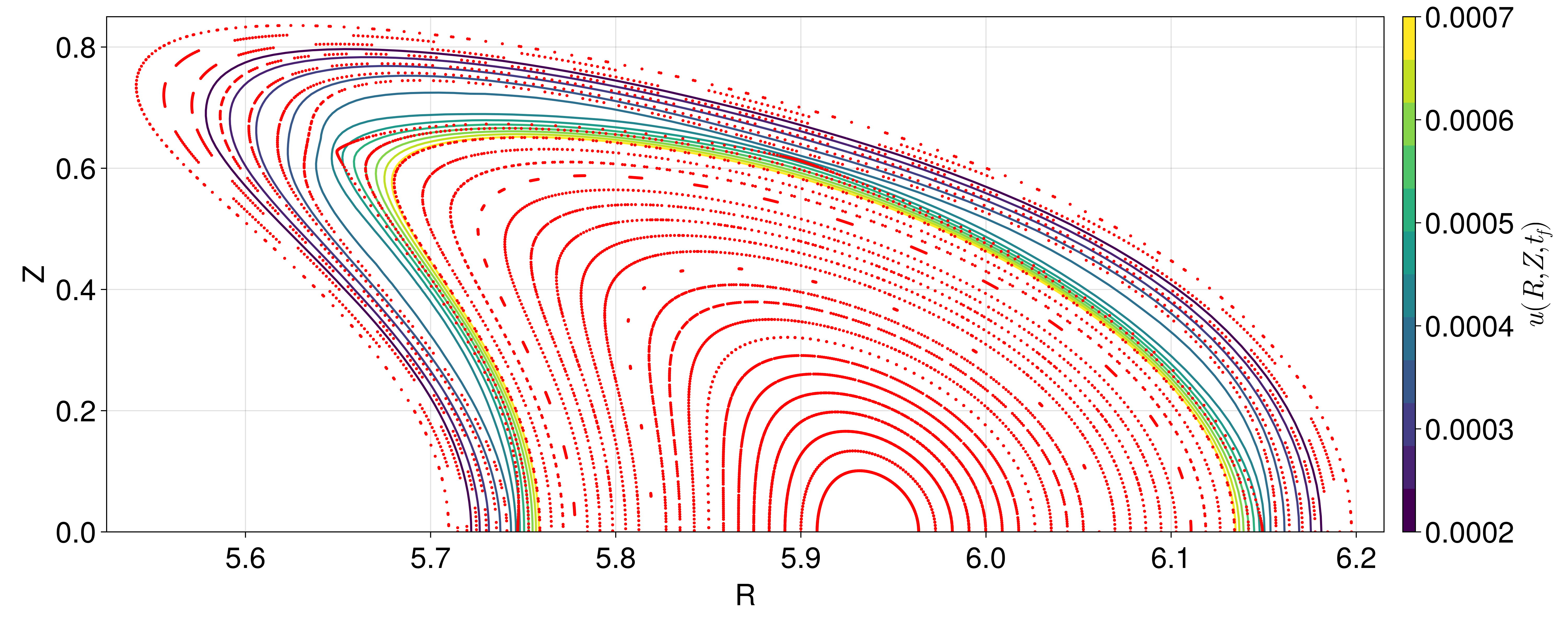}
    \caption{Contour plot showing the solution contours between $2\times10^{-4}\leq u\leq 7\times10^{-4}$ and the Poincar\'e section. This range has been selected to make it easier to identify the behaviour near the island. Note that we have significantly stretched the figure in the $R$ direction to make the contours easier to see.}
    \label{fig:SPEC equilibrium contour}
\end{figure}

It is possible that even smaller time steps would provide further flattening across the island. However, the wall-time may make this computationally infeasible. Furthermore, positivity of the solution at the boundary can become an issue with larger time steps. The simplest remedy to this is to use linear interpolation (instead of cubic Hermite) for the parallel map. However, this will also impact the order of the method.

\section{Conclusion}\label{sec:Conclusion}

In this work we have extended the penalty method for solving the anisotropic diffusion equation in Cartesian geometry to curvilinear coordinate systems introduced in \cite{muir_provably_2025}. 
We demonstrated how the method can be used to compute the solution in complex geometries by splitting the domain into multiple blocks, which are coupled together weakly by an interface simultaneous approximation term.
The summation by parts finite difference operators and parallel penalty gives a provably stable discretisation of the problem at the semi-discrete level. Furthermore, the unconditional stability property of the time stepping is also extended to the curvilinear case presented here.

Numerical tests by manufactured solution verify the convergence of the \ac{sbp} finite difference operator in curvilinear geometry.
Self convergence show that in complex geometry with the single island the convergence rate is degraded. We note that plots show that the error is maximum at the interior domain boundary, overshadowing the separatrix errors slightly.

Finally, we show that the method presented in this work is capable of modelling the solution profile in highly shaped geometries. We note that the time step in such a case is very limited. One remedy to this in the future may be to include more poloidal planes. We note also that the application of the parallel map is by far the most expensive part of the simulation, so improving efficiency would also be a possible way of resolving this issue.
Lastly, a positivity preserving interpolation scheme may be required to resolve issues at the edge of the simulation.

\section*{Acknowledgements}

\printbibliography

\appendix
\section{Symmetry and definiteness of perpendicular operator}\label{appendix:Perpendicular operator SPD}

We repeat the relevant matrix here for proof of the lemma below,
\begin{align*}
    \mathcal{A}_q = \begin{bmatrix}
        \tau_{q,0}          & [\bm{K}_q]_{j,k}    & [\bm{K}_{qr}]_{j,k}    \\
        [\bm{K}_q]_{j,k}    & [\bm{K}_q]_{j,k}    & [\bm{K}_{qr}]_{j,k}    \\
        [\bm{K}_{qr}]_{j,k} & [\bm{K}_{qr}]_{j,k} & [\bm{K}_r]_{j,k}
    \end{bmatrix}
\end{align*}

\begin{proof}[Proof of Lemma \ref{lem:SAT q Dirichlet matrix}]
    First define the matrices,
    \begin{align*}
        \widetilde{\mathcal{A}}_q = \begin{bmatrix} \tau_{q,0} - [\bm{K}_q]_{j,k} & 0 & 0 \\ 
        0 & [\bm{K}_q]_{j,k} & [\bm{K}_{qr}]_{j,k} \\
        0 & [\bm{K}_{qr}]_{j,k} & [\bm{K}_r]_{j,k} \end{bmatrix},\quad
        P = \begin{bmatrix} 1 & 1 & 0 \\ 0 & 1 & 0 \\ 0 & 0 & 1 \end{bmatrix},\quad
        C = \begin{bmatrix} [\bm{K}_q]_{j,k} & [\bm{K}_{qr}]_{j,k} \\
            [\bm{K}_{qr}]_{j,k} & [\bm{K}_r]_{j,k} \end{bmatrix},
    \end{align*}
    such that $P \widetilde{\mathcal{A}}_q P^T = \mathcal{A}_q$ and where $C$ is the lower right 2-by-2 block of $\mathcal{A}_q$. For any vector $\bm{v}\in\R^3$ we have,
    \begin{align*}
        \bm{v}^T P \widetilde{\mathcal{A}}_q P^T \bm{v} = 
        \bm{w}^T \widetilde{\mathcal{A}}_q \bm{w} = 
        w_1 (\tau_{q,0} - [\bm{K}_q]_{j,k}) w_1 + \begin{bmatrix} w_2 \\ w_3 \end{bmatrix}^T C \begin{bmatrix} w_2 \\ w_3 \end{bmatrix}.
    \end{align*}
    The first term is non-negative when $\tau_{q,0} \geq [\bm{K}_q]_{j,k}$, the second term is also non-negative if $C$ is positive semi-definite. Noting that $C$ is symmetric then by introducing
    \begin{align*}
        P_2 \begin{bmatrix} 1 & 0 \\ -\frac{[\bm{K}_{qr}]_{j,k}}{[\bm{K}_q]_{j,k}} & 1 \end{bmatrix} 
        \quad\text{such that},\quad 
        P_2 C P_2^T = \begin{bmatrix} [\bm{K}_q]_{j,k} & 0 \\ 0 & [\bm{K}_r]_{j,k} - \frac{[\bm{K}_{qr}]^2_{j,k}}{[\bm{K}_q]_{j,k}} \end{bmatrix},
    \end{align*}
    and noting the definitions in \eqref{eq:discrete curvilinear diffusion matrix} then the first diagonal of $P_2 C P^T$ is positive. Taking the bottom right term and dividing though by $J \kappa_\perp$ gives, for a given $j,k$,
    \begin{align*}
        (r_x^2 + r_y^2) - \frac{(q_x r_x + q_y r_y)^2}{q_x^2 + q_y^2} = \frac{(q_x r_x - q_y r_y)^2}{q_x^2 + q_y^2} \geq 0.
    \end{align*}
    Hence the matrix $C$ is also symmetric and positive semi-definite. Therefore when $\tau_{q,0} \ge [\bm{K}_q]_{j,k}$, the matrix $\mathcal{A}_q$ is symmetric and positive semi-definite.
\end{proof}

We repeat the relevant matrix here for proof of the lemma below,
\begin{align*}
    \mathcal{G} = \begin{bmatrix}
        \tau_{I,0} 					& \half [\bm{K}_q^-]_{n_q,i}			
            & \half [\bm{K}_q^+]_{1,i}	& \half [\bm{K}_{qr}^-]_{n_q,i} 	
            & \half [\bm{K}_{qr}^+]_{1,i} \\
        \half [\bm{K}_q^-]_{n_q,i}  	& [\bm{H}_q\bm{K}_q^-]_{n_q,i}	
            & 0                         & [\bm{H}_q\bm{K}_{qr}^-]_{n_q,i}
            & 0 \\
        \half [\bm{K}_q^+]_{1,i}    	& 0
            & [\bm{H}_q\bm{K}_q^+]_{1,i}& 0
            & [\bm{H}_q\bm{K}_{qr}^+]_{1,i} \\
        \half [\bm{K}_{qr}^-]_{n_q,i}   & [\bm{H}_q\bm{K}_{qr}^-]_{n_q,i}
            & 0                         & [\bm{H}_q\bm{K}_r^-]_{n_q,i} 	
            & 0 \\
        \half [\bm{K}_{qr}^+]_{1,i} 	& 0
            & [\bm{H}_q\bm{K}_{qr}^+]_{1,i} & 0
            & [\bm{H}_q\bm{K}_r^+]_{1,i}
    \end{bmatrix}
\end{align*}

\begin{proof}[Proof of Lemma \ref{lem:SAT interface matrix}]
    Consider the matrix,
    \begin{align*}
        P = \begin{bmatrix}
            1 & \frac{1}{2[\bm{H}_q]_{n_q,i}} & \frac{1}{2[\bm{H}_q]_{1,i}} & 0 & 0 \\
            0 & 1 & 0 & 0 & 0 \\
            0 & 0 & 1 & 0 & 0 \\
            0 & 0 & 0 & 1 & 0 \\
            0 & 0 & 0 & 0 & 1 
        \end{bmatrix}, \quad
        C = \begin{bmatrix}
            [\bm{H}_q\bm{K}_q^-]_{n_q,i} & 0 & 
                [\bm{H}_q\bm{K}_{qr}^-]_{n_q,i} & 0 \\
            0 & [\bm{H}_q\bm{K}_q^+]_{1,i} &
                0 & [\bm{H}_q\bm{K}_{qr}^+]_{1,i} \\
            [\bm{H}_q\bm{K}_{qr}^-]_{n_q,i} & 0 & 
                [\bm{H}_q\bm{K}_r^-]_{n_q,i} & 0 \\
            0 & [\bm{H}_q\bm{K}_{qr}^+]_{1,i} &
                0 & [\bm{H}_q\bm{K}_r^+]_{1,i}
        \end{bmatrix},
    \end{align*}
    \begin{align*}
        \widetilde{\mathcal{G}} = \begin{bmatrix} 
            \tau_{I,0} - \frac{[\bm{K}_q^-]_{n_q,i}}{[\bm{H}_q]_{n_q,i}} - \frac{[\bm{K}_q^+]_{1,i}}{[\bm{H}_q]_{1,i}} & 0 & 0 & 0 & 0 \\
            0 & [\bm{H}_q\bm{K}_q^-]_{n_q,i} & 0 & 
                [\bm{H}_q\bm{K}_{qr}^-]_{n_q,i} & 0 \\
            0 & 0 & [\bm{H}_q\bm{K}_q^+]_{1,i} &
                0 & [\bm{H}_q\bm{K}_{qr}^+]_{1,i} \\
            0 & [\bm{H}_q\bm{K}_{qr}^-]_{n_q,i} & 0 & 
                [\bm{H}_q\bm{K}_r^-]_{n_q,i} & 0 \\
            0 & 0 & [\bm{H}_q\bm{K}_{qr}^+]_{1,i} &
                0 & [\bm{H}_q\bm{K}_r^+]_{1,i}
        \end{bmatrix},
    \end{align*}
    such that $P\widetilde{\mathcal{G}}P^T = \mathcal{G}$ and the matrix $C$ is formed by the lower right 4-by-4 elements of $\mathcal{G}$. Then for any vector $\bm{v}\in\R^n$ we can write,
    \begin{align*}
        \bm{v}^T P\widetilde{\mathcal{G}}P^T \bm{v} = \bm{w}^T \widetilde{\mathcal{A}}P^T \bm{w} = 
            w_1 \left( -\frac{1}{4}\left( \frac{[\bm{K}_q^-]_{n_q,i}}{[\bm{H}_q]_{n_q,i}} + \frac{[\bm{K}_q^+]_{1,i}}{4[\bm{H}_q]_{1,i}} \right) + \tau_{I,0} \right) w_1 + 
            \begin{bmatrix} w_2 \\ w_3 \\ w_4 \\ w_5 \end{bmatrix}^T C \begin{bmatrix} w_2 \\ w_3 \\ w_4 \\ w_5 \end{bmatrix},
    \end{align*}
    where $C$ can be shown to be symmetric and positive semi-definite. Therefore the matrix $\mathcal{G}$ is symmetric and positive definite if $\tau_{I,0} \geq \frac{1}{4}\max ([\bm{K}_q^-]_{n_q,i}/[\bm{H}_q]_{n_q,i},[\bm{K}_q^+]_{1,i}/[\bm{H}_q]_{1,i})$.
\end{proof}

We now provide the proof of stability for the perpendicular operator.

\begin{proof}[Proof of Theorem \ref{thrm:perpendicular energy bound}]
    Consider the operator $\mathbf{D}_\perp$ of equation \eqref{eq:perpendicular FD operator}. 
    Multiplying $(\bm{P}_\perp)_I$ by $\bm{u}^T\bm{H}$ and expanding the cross terms using the SBP properties from \eqref{eq:curvilinear cross derivative sbp} 
    \begin{align*}
        \bm{u}^T \bm{H} (\bm{P}_\perp)_I \bm{u} &= 
        - \bm{u}^T \bm{H}_r \bm{D}_q^T \bm{H}_{qI} \bm{K}_q \bm{D}_q - \bm{u}^T \bm{H}_r \bm{B}_{qI} \bm{K}_q \bm{D}_q \bm{u} 
            \\&\qquad
        - \bm{u}^T \bm{H}_q \bm{D}_r^T \bm{H}_{rI} \bm{K}_r \bm{D}_r \bm{u} 
            \\&\qquad
        + \tau_{I,0} \bm{u}^T \bm{H}_r \hat{\bm{L}}_{q,0} \bm{u}
        + \tau_{I,1} \bm{u}^T \bm{H}_r \left( \hat{\bm{L}}_{q,1} \bm{K}_q \bm{D}_q + \hat{\bm{L}}_{q,1} \bm{K}_{qr} \bm{D}_r \right) \bm{u}
            \\&\qquad
        + \tau_{I,2} \bm{u}^T \bm{H}_r \left( \hat{\bm{L}}_{q,1} \bm{K}_q \bm{D}_q + (\bm{K}_{qr} \bm{D}_r)^T \hat{\bm{L}}_{q,1} \right) \bm{u} 
            \\&\qquad
        - \bm{u}^T \bm{D}_r^T \bm{H}_I \bm{K}_{qr} \bm{D}_q \bm{u} - \bm{u}^T \bm{D}_q^T \bm{H}_I \bm{K}_{qr} \bm{D}_r \bm{u} + \bm{u}^T \bm{H}_r \bm{B}_{qI} \bm{K}_{qr} \bm{D}_r \bm{u}.
    \end{align*}
    Setting $\tau_{I,1} = - \tau_{I,2} = \half$, introducing the matrix $\mathcal{G}$ from equation \ref{} and $\bm{v}$ from \ref{},
    \begin{align*}
        \bm{u}^T (\bm{A}_\perp)_I \bm{u} = 
            - \bm{u}^T \bm{H} \bm{P}_\perp \bm{u} &= -\sum_{} J \Delta r h_i \bm{v}^T \mathcal{G}_i \bm{v} \\
            &\qquad +\half \bm{u}^T \bm{H}_r \left( \hat{\bm{L}}^T_{q,1} \bm{K}_{qr} \bm{D}_r + (\bm{K}_{qr} \bm{D}_r)^T \hat{\bm{L}}_{q,1} \right)\bm{u} \\
            &\qquad \bm{u}^T \bm{H}_r \bm{R}_q \bm{u} + \bm{u}^T \bm{H}_q \bm{R}_r \bm{u}.
    \end{align*}
    This gives the desired estimate,
    \begin{align*}
        \bm{u}^T (\bm{A}_\perp)_I \bm{u} \ge 0,
    \end{align*}
    which completes the proof.
\end{proof}

\end{document}